%% file: Homotopy_invariance_constructible_sheaves_simplified.tex
\pdfoutput=1
\newif\ifpersonal
\personaltrue 
\documentclass[leqno,10pt]{amsart}
\usepackage[normal]{phaine_style}
\usepackage[margin=1.175in]{geometry}
\input{newcommands_all}
\input{document_macros}

\title{The homotopy-invariance of constructible sheaves}

\author{Peter J. Haine}
\address{Peter J. Haine, Department of Mathematics, University of California, Berkeley, Evans Hall, Berkeley, CA 94720, USA}
\email{phaine@math.berkeley.edu}

\author{Mauro Porta}
\address{Mauro Porta, Institut de Recherche Mathématique Avancée, 7 Rue René Descartes, 67000 Strasbourg, France}
\email{porta@math.unistra.fr}

\author{Jean-Baptiste Teyssier}
\address{Jean-Baptiste Teyssier, Institut de Mathématiques de Jussieu, 4 place Jussieu, 75005 Paris, France}
\email{jean-baptiste.teyssier@imj-prg.fr}

\date{\today}

\begin{document}

\begin{abstract} 
	The purpose of this paper is to explain why the functor that sends a stratified topological space $ S $ to the \category of constructible (hyper)sheaves on $ S $ with coefficients in a large class of presentable \categories is homotopy-invariant.
	To do this, we first establish a number of results in the unstratified setting, i.e., the setting of locally constant (hyper)sheaves.
	For example, if $X$ is a locally weakly contractible topological space and $\cE$ is a presentable \category, then we give a concrete formula for the constant hypersheaf functor \smash{$\cE \to \Shhyp(X;\cE)$}.
	This formula lets us show that the constant hypersheaf functor is a right adjoint, and is fully faithful if $ X $ is also weakly contractible.
	It also lets us prove a general monodromy equivalence and categorical Künneth formula for locally constant hypersheaves.
\end{abstract}

\maketitle

\tableofcontents


\setcounter{section}{-1}

\section{Introduction}

A classical result from sheaf theory says that the functor $\goesto{S}{\LC(S;\Set)}$ sending a topological space $ S $ to the category of locally constant sheaves of sets on $ S $ is homotopy-invariant.
More generally, if $ P $ is a poset then the functor
\begin{equation*}
	\goesto{S}{\ConsP(S;\Set)}
\end{equation*}
sending a $ P $-stratified topological space $ S $ to the category of sheaves of sets on $ S $ that are constructible with respect to the stratification $ \fromto{S}{P} $ is invariant under stratified homotopy equivalences.
Lurie's work on the topological exodromy equivalence (see \cite[Theorems \HAappthmlink{A.1.15} \& \HAappthmlink{A.4.19}]{HA}) generalizes these results by considering sheaves with values in the \category of \textit{spaces}, provided that we restrict to the following classes of well-behaved (stratified) topological spaces:
\begin{enumerate}
	\item For locally constant sheaves, we take topological spaces $S$ that are \emph{locally of singular shape}.
	
	\item For constructible sheaves, we take stratified topological spaces $S \to P$ for which the poset $P$ is \emph{Noetherian}, the stratification is \emph{conical}, $S$ is \emph{paracompact}, and all of the strata of $ S $ are \emph{locally of singular shape}.
\end{enumerate}
The goal of this paper is to establish the homotopy-invariance result for the \categories of locally constant and constructible sheaves with coefficients in the \category of spaces \emph{removing all of the above hypotheses}.

In the higher-categorical world, alongside with sheaves, it is often important to also consider hypersheaves.
Depending on the situation, one is better behaved than the other (see the discussion in \cite[\HTTsubsec{6.5.4}]{HTT}).
In the main body of the paper, we prove two versions of the homotopy-invariance theorem: one in the setting of hypersheaves and one in the setting of sheaves.
The hypersheaf one is stronger, requiring fewer assumptions than its sheaf-theoretic counterpart.
The precise statements are given later in this introduction, but the main advantages can be summarized as follows:
\begin{enumerate}
	\item Working with hypersheaves, we establish invariance not only with respect homotopy equivalences, but to a large class of \textit{weak} homotopy equivalences (a result that seems new even for sheaves of \textit{sets}).
	
	\item Working with hypersheaves, we can drop the Noetherianity assumption on the poset $P$.
\end{enumerate}
We expect both of these statements to fail in the sheaf-theoretic setting.
Furthermore, both facts have interesting consequences.
The first is needed, \emph{at this level of generality}, in the companion paper of Porta--Teyssier \cite{PortaTeyssier} concerning a strengthening of the exodromy equivalence of \HAa{Theorem}{A.9.3}.
The second lets us apply the homotopy-invariance theorem to key examples like infinite Grassmannians or the \textit{Ran space} of a manifold \cites[\HTTsubsec{5.5.1}]{HTT}[\S3.7]{MR3590534}{arXiv:1910.11980}{arXiv:2107.11243}{arXiv:2012.08504}, whose natural stratification is not Noetherian.
This was one of the motivations behind Lejay's work on the exodromy equivalence \cite{arXiv:2102.12325}.

Finally, we do not limit ourselves to sheaves of spaces.
Rather, our results apply to more general presentable (not necessarily compactly generated) \categories: the methods of this paper are explicit enough that we can handle any stable presentable \category, and any \topos without any added difficulty.


\subsection*{Statement of results}\label{subsec:statementofresults}

Before giving the precise statements of the main homotopy-invariance results of this paper, let us be precise about what we mean by \textit{homotopy-invariance}.
Fix a poset $P$, that we regard as a topological space via the Alexandroff topology (where the open subsets are the upwards-closed subsets, see \Cref{notation:Alexandroff_topology}).
A \textit{$ P $-stratified topological space} is the data of a topological space $S$ together with a continuous map $S \to P$.
When $P = \pt$, a $P$-stratified space is just a topological space.
Given \smash{$S \in \TopP$} and $X \in \Top$, we regard $S \cross X$ as a $P$-stratified space via the projection $S \cross X \to S \to P$.
Consider the following definition:

\begin{definition}\label{def:homotopy_invariant}
	Let $ P $ be a poset. 
	A functor \smash{$ C \colon \fromto{\TopP^{\op}}{\Catinfty} $} is \defn{homotopy-invariant} if for each $ P $-stratified space $ S $, the functor
	\begin{equation*}
		C(\pr_{S}) \colon \fromto{C(S)}{C(S \cross [0,1])}
	\end{equation*}
	induced by the projection $\pr_S \colon S \cross [0,1] \to S$ is an equivalence of \categories.
	A functor \smash{$ C \colon \fromto{\TopP^{\op}}{\Catinfty} $} is \defn{strongly homotopy-invariant} if for each $P$-stratified space $S$ and each weakly contractible and locally weakly contractible\footnote{Starting from here on we shorten ``weakly contractible and locally weakly contractible'' to \textit{\wclwc}.} topological space $X$, the induced functor
	\begin{equation*} 
		C(\pr_S) \colon \fromto{C(S)}{C(S \cross X)}
	\end{equation*}
	is an equivalence of \categories.
\end{definition}

Let $ \E $ be a presentable \category and $ S $ a topological space.
We write $ \LC(S;\E) $ for the \category of locally constant $ \E $-valued sheaves on $ S $, and write \smash{$ \LChyp(S;\E) $} for the hypersheaf variant of this \category (see \cref{subsec:sheaveshypersheaves,subsec:background_locally_constant}).
Beware that, in general, the notions of local constancy for sheaves and hypersheaves are not the same and \smash{$ \LChyp(S;\E) $} is not a subcategory of $ \LC(S;\E) $. 

\begin{theorem}[(\Cref{thm:homotopy_invariance_hyperconstant} \& \Cref{cor:hiofLC})]\label{thm:introhiofLC}
	The functors
	\begin{equation*}
		\LC(-;\E), \: \LChyp(-;\E) \colon \fromto{\Top^{\op}}{\Catinfty} 
	\end{equation*}
	are homotopy-invariant.
	Moreover, $\LChyp(-;\E)$ is strongly homotopy invariant
\end{theorem}

\noindent Passing to global sections, \Cref{thm:introhiofLC} implies that cohomology with coefficients in a locally constant sheaf valued in any presentable \category is homotopy-invariant.

\begin{remark}
	The same kind of techniques involved in the proof of \Cref{thm:introhiofLC} allow to show that the functor \smash{$\LChyp(-;\E)$} inverts all weak homotopy equivalences between locally weakly contractible topological spaces (see \cref{obs:monodromy}).
	However, the functors $ \LC(-;\E) $ and \smash{$ \LChyp(-;\E) $} do \textit{not} invert all weak homotopy equivalences between \emph{arbitrary} topological spaces: sheaf cohomology with constant coefficients is not an invariant of the weak homotopy type of a topological space.
	Indeed, for paracompact spaces, Čech cohomology and sheaf cohomology agree \cite[Théorème 5.10.1]{MR0102797}.
	Now, the Warsaw circle is weakly contractible, but the quotient map from it to the circle induces an isomorphism on Čech cohomology, hence sheaf cohomology.
	Note that this doesn't fall into the setting of \Cref{thm:introhiofLC}: the Warsaw circle is not even locally path-connected.
\end{remark}

Fix a poset $P$.
Given a $P$-stratified topological space $ \fromto{S}{P} $, we write $ \ConsP(S;\E) $ for the \category of constructible $ \E $-valued sheaves on $ S $, and write \smash{$ \ConsPhyp(S;\E) $} for the hypersheaf variant of this \category (see \cref{subsec:stratifiedspaces} for precise definitions).
Since constructible sheaves are locally constant along a stratification, as long as the poset $ P $ and coefficients $ \E $ allow to check equivalences after pulling back to strata, then \Cref{thm:introhiofLC} implies that constructible sheaves are homotopy-invariant.
We offer two ways of checking this:

\begin{theorem}[(\Cref{cor:hiofConsCG,cor:hiofConsnoeth})]\label{thm:mainintro}
	Consider the functors
	\begin{equation*}
		\ConsP(-;\E), \: \ConsPhyp(-;\E) \colon \fromto{\TopP^{\op}}{\Catinfty} \period
	\end{equation*}
	\begin{enumerate}[label=\stlabel{thm:mainintro}, ref=\arabic*]
		\item\label{thm:mainintro.1} If $ \E $ is compactly generated, then the functor $\ConsPhyp(-;\E)$ is strongly homotopy-invariant.

		\item\label{thm:mainintro.2} If $ P $ is Noetherian and $ \E $ is compactly generated, stable, or \atopos, then the functor $\ConsP(-;\E)$ is homotopy-invariant, and \smash{$\ConsPhyp(-;\E)$} is strongly homotopy-invariant.
	\end{enumerate}
\end{theorem}	

\begin{remark}
	\Cref{thm:mainintro} holds under much weaker assumptions than Lurie's exodromy equivalence \HAa{Theorem}{A.9.3}.
	For instance, it holds for stratified spaces that are not necessarily conical.
\end{remark}

\noindent Again, passing to global sections, \Cref{thm:mainintro} implies that (under the above hypotheses) sheaf cohomology with coefficients in a constructible sheaf is homotopy-invariant.
Also note that \Cref{thm:mainintro} generalizes the following existing results about the homotopy-invariance of constructible sheaves:
\begin{enumerate}
	\item In the setting of topologically stratified spaces in the sense of Goresky--MacPherson \cite[\S1.1]{MR696691}, Treumann showed that constructible sheaves with values in the $ 2 $-category $ \Cat_{1} $ of $ 1 $-categories is homotopy-invariant \cite[Theorem 3.11]{MR2575092}.

	\item When $ P $ is Noetherian, Clausen--Ørsnes Jansen proved that $ \ConsP(-;\Spc) $ is homotopy-invariant \cite[Proposition 3.2]{arXiv:2108.01924}.
	Our proof for $ \ConsP(-;\E) $ is a mild extension of their work.
\end{enumerate}

One of the key ingredients of the proofs of \Cref{thm:introhiofLC,thm:mainintro} is the notion of \emph{topological family of locally hyperconstant hypersheaves}.
Concretely, if $X$ is a \wclwc topological space (e.g.\ $X = [0,1]$), and $S$ is any topological space, we are led to consider the full subcategory
\begin{equation*}
	\LChyp_S(S \cross X; \E) \subset \Shhyp(S \cross X; \E)
\end{equation*}
spanned by those hypersheaves that are, locally on $X$, pulled back from hypersheaves on $S$ (see \cref{def:LC} for the precise definition).
When $\E = \Spc$, this definition recovers the usual notion of \textit{foliated} hypersheaf (see \cref{subsec:foliated_hypersheaves}), but it is better behaved for general coefficients.
The main theorem concerning these objects is the following:

\begin{theorem}[(\Cref{prop:computing_hyperconstant_hypersheaf,thm:relative_full_faithfulness})]\label{thm_intro:full_faithfulness}
	Let $X$ be a locally weakly contractible topological space and let $\E$ be a presentable $\infty$-category.
	Then for every topological space $S$, the pullback functor
	\begin{equation*} 
		\prSupperstarhyp \colon \Shhyp(S; \E) \longrightarrow \Shhyp(S \cross X; \E) 
	\end{equation*}
	admits a left adjoint.
	Moreover, if $ X $ is also weakly contractible, then \smash{$ \prSupperstarhyp $} is fully faithful with essential image \smash{$\LChyp_S(S \cross X; \E)$}.
\end{theorem}

It is easy to explain the idea behind this theorem when $S = \pt$ and $\E = \Spc$.
Let $\Gamma_X \colon X \to \pt$ be the unique map, and let
\begin{equation*} 
	\Pi_\infty \colon \Shhyp(X) \longrightarrow \Spc 
\end{equation*}
be the left Kan extension of the functor sending an open $U \subset X$ to its underlying homotopy type $\Pi_\infty(U)$.
For formal reasons, $\Pi_\infty$ admits a right adjoint, that we denote $\Pi^\infty$.
Given $K \in \Spc$, the hypersheaf $\Pi^\infty(K) $ is given by the assignment
\begin{equation*} 
	U \mapsto \Pi^\infty(K)(U) \coloneqq \Map_{\Spc}(\Pi_\infty(U),K) \period 
\end{equation*}
There is a natural comparison map $\GammaXupperstarhyp \to \Pi^\infty$, and the fact that we are working in the hypercomplete setting and that $X$ is locally weakly contractible implies that this map is an equivalence (see \cref{prop:computing_hyperconstant_hypersheaf}).
Note that if $ X $ is also weakly contractible, then the full faithfulness part follows from the assumption that $\Pi_\infty(X) \simeq \pt $.

Besides \Cref{thm:introhiofLC,thm:mainintro}, \Cref{thm_intro:full_faithfulness} has many other consequences; we discuss them in \cref{sec:consequences_of_full_faithfulness}.
Among them are: a general form of the monodromy equivalence (\Cref{cor:monodromy}), a Künneth formula for locally hyperconstant hypersheaves (\Cref{cor:Kunneth}), and the comparison between sheaf and singular cohomology on locally weakly contractible spaces (\Cref{cor:cohomology_comparison}).
We also establish the following handy recognition criterion:

\begin{corollary}[(\Cref{prop:strong_local_hyperconstancy})]\label{cor_intro:strong_local_hyperconstancy}
	Let $X$ be a locally weakly contractible topological space and let $\E$ be a presentable \category.
	For a hypersheaf \smash{$F \in \Shhyp(X;\E)$}, the following statements are equivalent:
	\begin{enumerate}[label=\stlabel{cor_intro:strong_local_hyperconstancy}, ref=\arabic*]
		\item The hypersheaf $F$ is locally hyperconstant.
		
		\item For every pair of weakly contractible open subsets $U \subset V$ of $X$, the restriction map $ F(V) \to F(U) $ is an equivalence in $\E$.
	\end{enumerate}
\end{corollary}

\noindent In particular, it immediately follows that if $X$ is locally weakly contractible, then locally hyperconstant hypersheaves are closed under arbitrary limits in \smash{$\Shhyp(X;\E)$}.


\subsection*{Linear overview}\label{subsec:linoverview}

\Cref{sec:sheaf_background} recalls background from sheaf theory.
In \cref{sec:hyperconstant}, we prove \Cref{thm_intro:full_faithfulness}.
\Cref{subsec:constant_hypersheaf} provides an alternative construction of the hyperconstant hypersheaf functor; in \cref{subsec:exceptional_pushforward}, we use this characterization to introduce the exceptional pushforward and study its properties.
In \cref{subsec:full_faithfulness} we establish \Cref{thm_intro:full_faithfulness}, and in \cref{subsec:foliated_hypersheaves} we discuss the relationship between the \category $\LC_S(S \cross X;\E)$ and the \category of \textit{foliated} hypersheaves.
\Cref{sec:consequences_of_full_faithfulness} is dedicated to the many consequences of \Cref{thm_intro:full_faithfulness}: in \cref{subsec:structural_results} we deduce many unexpected categorical properties of $\LC_S(S \cross X; \E)$; in \cref{subsec:monodromy} we establish a general form of the monodromy equivalence, as well as a Künneth formula for the \category of locally hyperconstant hypersheaves. 
In \cref{subsec:HIforLChypersheaves} we deduce the hypersheaf part of \Cref{thm:introhiofLC}, and in \cref{subsec:cohomology_comparison} we obtain a comparison result between sheaf and singular cohomology.
In \Cref{sec:homotopy_invariance_for_LC} we partially adapt these results to the setting of sheaves, first establishing the existence of the exceptional pushforward under certain additional assumptions, and then proving the sheaf part of \Cref{thm:introhiofLC}.
\Cref{sec:homotopy_invariance_for_constructible} is dedicated to the homotopy-invariance results for stratified spaces.
To begin with, we establish a ``formal'', unconditional version: for every stratified space $S \to P$, any presentable \category $\E$ and any \wclwc topological space $X$, we identify the essential image of the fully faithful the pullback functor
\begin{equation*} 
	\prSupperstarhyp \colon \ConsPhyp(S; \E) \inclusion \ConsPhyp(S \cross X; \E) 
\end{equation*}
(\cref{thm:formal_homotopy_invariance}).
We also provide several criteria to establish its essential surjectivity (\cref{cor:formal_homotopy_invariance}).
Finally, in \cref{subsec:detecting_equivalences_on_strata} we show how various combinations of assumptions on $\E$ and on $P$ can be used to verify one of these criteria.

\begin{acknowledgments}
	We greatly benefited from conversations with Guglielmo Nocera and Marco Volpe; it is a pleasure to thank them.
	We thank Dustin Clausen for explaining why we can't just work with hypersheaves that are constructible in the usual sense.
	We are indebted to Dustin Clausen and Mikala Ørsnes Jansen for their proof of \cite[Proposition 3.2]{arXiv:2108.01924}; expanding on their proof helped us prove these results in a much cleaner way than we initially had.
	We thank Kiran Luecke for asking a question that led to a great simplification of some of the material presented here.

	PH gratefully acknowledges support from the MIT Dean of Science Fellowship, the NSF Graduate Research Fellowship under Grant \#112237, UC President's Postdoctoral Fellowship, and NSF Mathematical Sciences Postdoctoral Research Fellowship under Grant \#DMS-2102957. 
\end{acknowledgments}


\section{Sheaf-theoretic background}\label{sec:sheaf_background}

The purpose of this section is to explain our sheaf-theoretic conventions and notation as well as recall some background on hypersheaves. 
We recall the basics in \cref{subsec:sheaveshypersheaves}.
In \cref{subsec:hypercovbases}, we recall the interaction between hypersheaves and bases for Grothendieck topologies.
In \cref{subsec:background_locally_constant}, we provide background on locally (hyper)constant (hyper)sheaves.


\subsection{Background on sheaves \& hypersheaves}\label{subsec:sheaveshypersheaves}

Throughout this subsection, we fix \asite $(\cC,\tau)$ and a presentable \category $\E$.

\begin{notation}
	We write
	\begin{equation*} 
		\PSh(\cC;\E) \coloneqq \Fun(\cC^{\op}, \E) 
	\end{equation*}
	for the \category of $\E$-valued presheaves on $ \cC $.
	We also write $\Sh_{\tau}(\cC;\E) \subset \PSh(\cC;\E) $ for the full subcategory spanned by $\E$-valued presheaves that satisfy $\tau$-descent.
	When $\E = \Spc$, we simply write
	\begin{equation*}
		\PSh(\cC) \colonequals \PSh(\cC;\Spc) \andeq \Sh_{\tau}(\cC) \colonequals \Sh_{\tau}(\cC;\Spc) \period
	\end{equation*}
\end{notation}	

\begin{nul}
	The \categories $\PSh(\cC;\E)$ and $\Sh_{\tau}(\cC;\E)$ are naturally identified with the tensor products of presentable \categories $\PSh(\cC) \tensor \E$ and $\Sh_{\tau}(\cC) \tensor \E $ \cite[\SAGthm{Remark}{1.3.1.6} \& \SAGthm{Proposition}{1.3.1.7}]{SAG}.
	We refer the reader to \cite[\HAsubsec{4.8.1}]{HA} for a thorough treatment of the tensor product of presentable \categories.
	As both points of view have their own advantages, in this paper we use both descriptions interchangeably.
\end{nul}

\begin{nul}\label{nul:hypersheaves}
	Crucial to the current paper is the notion of \textit{hypersheaf}.
	When $\E$ is the \category of spaces, hypersheaves can be defined intrinsically in the \topos $\Sh_{\tau}(\cC)$ as \defn{hypercomplete objects}, that is, objects that are local with respect to $ \infty $-connected morphisms.
	Hypersheaves thus form a full subcategory \smash{$\Shhyp_{\tau}(\cC) \subset \Sh_{\tau}(\cC)$}.
	It is then possible to \emph{define} hypersheaves with coefficients in $\E$ as the tensor product
	\begin{equation*} 
		\Shhyp_{\tau}(\cC;\E) \coloneqq \Shhyp_{\tau}(\cC) \tensor \E \period 
	\end{equation*}
	Each of the inclusions
	\begin{equation*} 
		\Shhyp_{\tau}(\cC) \subset \PSh(\cC) \andeq \Shhyp_{\tau}(\cC) \subset \Sh_{\tau}(\cC) 
	\end{equation*}
	admits a left adjoint.
	We refer to both left adjoints as the \defn{hypercompletion functors}, and we denote them by $(-)^{\hyp}$.
	Functoriality of the tensor product of presentable \categories produces functors
	\begin{equation*} 
		(-)^{\hyp} \colon \PSh(\cC;\E) \longrightarrow \Shhyp_{\tau}(\cC;\E) \andeq (-)^{\hyp} \colon \Sh_{\tau}(\cC;\E) \longrightarrow \Shhyp_{\tau}(\cC;\E) \period 
	\end{equation*}
	Both these functors still admit fully faithful right adjoints.
	We refer the reader unfamiliar with hypercomplete objects and hypercompletion to \cite[\S\S \HTTsubseclink{6.5.2}--\HTTsubseclink{6.5.4}]{HTT} or \cite[\S 3.11]{arXiv:1807.03281} for further reading on the subject.
\end{nul}

\begin{nul}
	If there exists an integer $ n \geq 0 $ such that $ \E $ is an $ n $-category, then $ \Shhyp_{\tau}(\cC;\E) = \Sh_{\tau}(\cC;\E) $ \cites[\HTTthm{Lemma}{6.5.2.9}]{HTT}[\HAthm{Example}{4.8.1.22}]{HA}.
	In particular, every sheaf of sets is a hypersheaf.
\end{nul}

\begin{notation}\label{nul:topological_space}
	Let $ S $ be a topological space.
	We write $ \Open(S) $ the poset of open subsets of $ S $, ordered by inclusion.
	We regard $ \Open(S) $ as a site with the covering families given by open covers.
	We write
	\begin{equation*}
		\PSh(S;\E) \colonequals \PSh(\Open(S);\E), \quad \Sh(S;\E) \colonequals \Sh(\Open(S);\E), \quad \text{and} \quad \Shhyp(S;\E) \colonequals \Shhyp(\Open(S);\E) \period
	\end{equation*}
\end{notation}

\begin{nul}\label{nul:CWhypercomplete}
	Sheaves and hypersheaves on topological spaces coincide in many situations in which homotopy-invariance is a well-behaved notion.
	For example, the \topos of sheaves on a topological space admitting a CW structure is hypercomplete \cite{MO:168526}.
\end{nul}

\begin{recollection}\label{rec:stalksconservativeCGhyp}
	Let $ S $ be a topological space.
	Then the stalk functors
	\begin{equation*}
		\{\supperstar \colon \fromto{\Shhyp(S)}{\Sh(\{s\}) \equivalent \Spc} \}_{s \in S}
	\end{equation*}
	are jointly conservative \HAa{Lemma}{A.3.9}.
	Since the stalk functors are left exact, \cite[Lemma 2.8]{arXiv:2108.03545} shows that for any compactly generated \category $ \E $, the stalk functors
	\begin{equation*}
		\{\supperstar \colon \fromto{\Shhyp(S;\E)}{\Sh(\{s\};\E) \equivalent \E} \}_{s \in S}
	\end{equation*}
	are jointly conservative.
\end{recollection}

\begin{nul}\label{nul:Shhyp_via_stalks}
	Let $ S $ be a topological space and $ \E $ a compactly generated \category.
	Then the subcategory
	\begin{equation*}
		\Shhyp(S;\E) \subset \Sh(S;\E)
	\end{equation*}
	is the localization obtained by inverting all morphisms that induce equivalences on stalks.
\end{nul}

\begin{nul}
	Let $ R $ be a ring and write $ \Dup(R) $ for the derived \category of $ R $.
	Since $ \Dup(R) $ is compactly generated, as a special case of \Cref{nul:Shhyp_via_stalks}, the \category \smash{$ \Shhyp(S;\Dup(R)) $} is the \categorical enhancement of the classical unbounded derived category of sheaves of $ R $-modules on $ S $.
\end{nul}

\subsection{Hypersheaves and bases}\label{subsec:hypercovbases}

In the rest of the paper, we make a through use of bases for $\infty$-sites:

\begin{definition}
	Let $(\cC,\tau)$ be \asite.
	A \emph{basis} of $(\cC, \tau)$ is a full subcategory $\cB$ of $\cC$ such that every object $U \in \cC$ admits a $\tau$-covering $\{U_i\}_{i \in I}$ where for each $ i \in I $, we have $U_i \in \cB$.
\end{definition}

\begin{example}\label{eg:contractible_opens_are_a_basis}
	Let $S$ and $X$ be topological spaces.
	Write
	\begin{equation*}
		\Opencross(S \cross X) \subset \Open(S \cross X)
	\end{equation*}
	for the subposet spanned by the open subsets of the form $V \cross U$, where $V \in \Open(S)$ and $U \in \Open(X)$.
	Then $\Opencross(S \cross X)$ is a basis of $\Open(S \cross X)$.
	
	We write
	\begin{equation*}
		\Openctrall(S \cross X) \subset \Opencross(S \cross X)
	\end{equation*}
	for the subposet spanned by the open subsets of the form $V \cross U$, where $U$ is a weakly contractible open subset of $X$.
	When $S = \pt$, we simply write $\Openctr(X)$ instead of $\Openctrall(\pt \cross X)$.

	If $ X $ is locally weakly contractible, then $\Openctrall(S \cross X)$ is also basis of $\Open(S \cross X)$.
\end{example}

Let $(\cC,\tau)$ be \asite and $\cB$ be a basis for $(\cC,\tau)$.
Write $j \colon \cB^{\op} \inclusion \cC^{\op}$ for the inclusion.
Right Kan extension along $j$ defines a fully faithful functor
\begin{equation*} 
	\jlowerstar \colon \PSh(\cB;\E) \inclusion \PSh(\cC;\E)  
\end{equation*}
with left adjoint $ \jupperstar \colon \PSh(\cC;\E) \to \PSh(\cB;\E) $ given by restriction of presheaves.

\begin{definition} \label{def:basis_hypersheaf}
	We say that an $\E$-valued presheaf $F \in \PSh(\cB;\E)$ on $\cB$ is a \defn{$\tau$-hypersheaf} if $\jlowerstar(F)$ belongs to \smash{$\Shhyp_{\tau}(\cC;\E)$}.
	We write
	\begin{equation*}
		\Shhyp_{\tau}(\cB;\E) \subset \PSh(\cB;\cE)
	\end{equation*}
	for the full subcategory spanned by $\tau$-hypersheaves.
\end{definition}

The key fact we need is that hypersheaves on a site and a basis agree:

\begin{proposition}[{\cites[Theorem A.6]{arXiv:2001.00319}[Proposition 3.12.11]{arXiv:1807.03281}}]\label{prop:hypersheaves_as_basis_RKE}
	Let $(\cC,\tau)$ be \asite and $\cB \subset \cC $ a basis. 
	Then:
	\begin{enumerate}[label=\stlabel{prop:hypersheaves_as_basis_RKE}, ref=\arabic*]
		\item\label{prop:hypersheaves_as_basis_RKE.1} For every $F \in \Shhyp_{\tau}(\cC;\E)$, the unit transformation $\unit \colon F \to \jlowerstar \jupperstar(F)$ is an equivalence.
		
		\item\label{prop:hypersheaves_as_basis_RKE.2} The functor $\jlowerstar \colon \Shhyp_\tau(\cB;\E) \to \Shhyp_\tau(\cC;\E)$ is an equivalence with inverse given by the presheaf-theoretic restriction $\jupperstar$.
	\end{enumerate}
\end{proposition}

\begin{remark}\label{rem:hypersheafification}
	Let $F \in \PSh(\cC;\E)$.
	It follows directly from \Cref{prop:hypersheaves_as_basis_RKE} that if $j^\ast(F)$ is a hypersheaf in the sense of \Cref{def:basis_hypersheaf}, then the unit $F \to j_\ast j^\ast(F)$ exhibits $j_\ast j^\ast(F)$ as hypersheafification of $F$.
\end{remark}


\subsection{Background on locally (hyper)constant sheaves}\label{subsec:background_locally_constant}

In what follows we limit our discussion of locally (hyper)constancy and the functoriality of sheaves to the setting of topological spaces.
Throughout this subsection, we fix a presentable \category $\E$.

\begin{recollection}
	Let $ f \colon \fromto{X}{Y} $ be a map of topological spaces.
	We write
	\begin{equation*}
		\flowerstar \colon \fromto{\PSh(X;\E)}{\PSh(Y;\E)}
	\end{equation*}
	for the \defn{pushforward} functor defined by the formula
	\begin{equation*}
		\flowerstar(F)(V) \colonequals F(\finverse(V)) \period
	\end{equation*}
	Recall that the pushforward functor $ \flowerstar $ carries sheaves to sheaves and hypersheaves to hypersheaves (for the latter statement, see the proof of \HTT{Proposition}{6.5.2.13}). 

	We write
	\begin{equation*}
		\finverse \colon \fromto{\PSh(Y)}{\PSh(X)}
	\end{equation*}
	for \defn{presheaf pullback} functor; $ \finverse $ is the left adjoint to $ \flowerstar \colon \fromto{\PSh(Y;\E)}{\PSh(X;\E)} $.
	In general, the functor $ \finverse $ preserves neither sheaves nor hypersheaves.
	We write 
	\begin{equation*}
		\fupperstar \colon \fromto{\Sh(Y;\E)}{\Sh(X;\E)} \andeq \fupperstarhyp \colon \fromto{\Shhyp(Y;\E)}{\Shhyp(X;\E)}
	\end{equation*}
	for the composites of $ \finverse \colon \fromto{\Sh(Y;\E)}{\PSh(X;\E)} $ with (hyper)sheafification.
	It follows formally that $\fupperstarhyp \simeq (-)^{\hyp} \circ \fupperstar$.
	By construction, there are adjunctions $\fupperstar \leftadjoint \flowerstar$ and $\fupperstarhyp \leftadjoint \flowerstar$.
\end{recollection}

\begin{notation}
	If $ f \colon \incto{X}{Y} $ is the inclusion of a subspace, we simply write 
	\begin{equation*}
		\restrict{(-)}{X} \colonequals \fupperstar \andeq \restrict{(-)}{X}^{\hyp} \colonequals \fupperstarhyp \period
	\end{equation*}
	If the space-valued sheaf pullback functor $ \fupperstar \colon \fromto{\Sh(Y)}{\Sh(X)} $ admits a left adjoint, then for every presentable \category $ \E $, the pullback functor $ \fupperstar \colon \fromto{\Sh(Y;\E)}{\Sh(X;\E)} $ carries hypersheaves to hypersheaves \HAa{Lemma}{A.2.6}.
	In particular, if $ U \subset Y $ is an open subset, then the functor $ \restrict{(-)}{U} \colon \fromto{\Sh(Y;\E)}{\Sh(U;\E)} $ carries hypersheaves to hypersheaves.
\end{notation}

\begin{notation}
	Let $S$ and $X$ be topological spaces.
	We denote by
	\begin{equation*} 
		\pr_S \colon S \cross X \longrightarrow S \andeq \pr_X \colon S \cross X \longrightarrow X 
	\end{equation*}
	the projections.
	When $S = \pt$ we write $\Gamma_X$ instead of $\pr_\pt$.
	Thus $ \GammaXlowerstar \colon \fromto{\Sh(X;\E)}{\Sh(*;\E) \equivalent \E} $ is the global sections functor and $ \Gamma_{X}\inv $ is the constant \textit{pre}sheaf functor.
	Moreover, the functors
	\begin{equation*}
		\GammaXupperstar \colon \fromto{\E}{\Sh(X;\E)} \andeq \GammaXupperstarhyp \colon \fromto{\E}{\Shhyp(X;\E)}
	\end{equation*}
	are the constant sheaf and hypersheaf functors, respectively.
	Analogously, for every $S$ we refer to the functor $\pr_S\inv$ (resp.\ $\prSupperstar$, \smash{$\prSupperstarhyp$}) as the \emph{$S$-constant presheaf} (resp. \emph{sheaf}, \emph{hypersheaf}) \emph{functor}.
\end{notation}

\begin{definition}\label{def:LC}
	Let $ S $ and $ X $ be topological spaces and let $\E$ be a presentable \category.
	\begin{enumerate}[label=\stlabel{def:LC}, ref=\arabic*]
		\item We say that a sheaf $ L \in \Sh(S \cross X;\E) $ is \defn{constant relative to $ S $} (or \defn{$S$-constant}) if $ L $ is in the essential image of the $S$-constant sheaf functor $ \prSupperstar \colon \fromto{\Sh(S;\E)}{\Sh(S \cross X;\E)} $.
		
		\item We say that $ L \in \Sh(S \cross X;\E) $ is \defn{locally constant relative to $S$} (or \defn{locally $S$-constant}) if there exists an open cover $ \{U_{\alpha}\}_{\alpha \in A} $ of $ X $ such that for each $ \alpha \in A $, the restriction $ \restrict{L}{S \cross U_{\alpha}} $ is a $S$-constant sheaf on $ S \cross U_{\alpha} $.
		
		\item We say that a hypersheaf $ L \in \Shhyp(S \cross X;\E) $ is \defn{hyperconstant relative to $S$} (or \defn{$S$-hyperconstant}) if $ L $ is in the essential image of the constant hypersheaf functor \smash{$ \prSupperstarhyp \colon \fromto{\Shhyp(S;\E)}{\Shhyp(S \cross X;\E)} $}.
		
		\item We say that $ L \in \Shhyp(S \cross X;\E) $ is \defn{locally hyperconstant relative to $S$} (or \defn{locally $S$-hyperconstant}) if there exists an open cover $ \{U_{\alpha}\}_{\alpha \in A} $ of $ X $ such that for each $ \alpha \in A $, the restriction $ \restrict{L}{S \cross U_{\alpha}} $ is a $S$-hyperconstant \textit{hyper}sheaf on $ S \cross U_{\alpha} $.
	\end{enumerate}
	We write
	\begin{equation*}
		\LC_S(S \cross X;\E) \subset \Sh(S \cross X;\E) \andeq \LChyp_S(S \cross X;\E) \subset \Shhyp(S \cross X;\E)
	\end{equation*}
	for the full subcategories spanned by the locally $S$-constant sheaves and locally $S$-hyperconstant hypersheaves, respectively.
	When $S = \pt$ we denote these \categories by $\LC(X;\E)$ and \smash{$\LChyp(X;\E)$}, respectively.
\end{definition}

\begin{warning}\label{warning:LChyp}
	We emphasize that for a given object $ E \in \E $, the constant sheaf $ \GammaXupperstar(E) $ need not be hypercomplete.
	Similarly, a hyperconstant hypersheaf need not be a constant sheaf; the notions of constant sheaves and hyperconstant hypersheaves are genuinely different.
	Also notice that there is a containment
	\begin{equation*}
		\LC(X;\E) \intersect \Shhyp(X;\E) \subset \LChyp(X;\E) \period
	\end{equation*}
	However, this inclusion is not generally an equality.
\end{warning}

\begin{remark}\label{rem:LChyp_and_LC_agree}
	If $ X $ is a topological space \textit{locally of singular shape} in the sense of \HAa{Definition}{A.4.15}, then 
	\begin{equation*}
		\LChyp(X) = \LC(X) \intersect \Shhyp(X) = \LC(X) \period
	\end{equation*}
	See \cites[\HAappthm{Corollary}{A.1.17}]{HA}[Proposition 2.1]{arXiv:2102.12325}.
\end{remark}

\begin{observation}\label{obs:pullbackofhypLC}
	Let $ S $ be a topological space and $ f \colon \fromto{X}{Y} $ a map of topological spaces.
	Write $f_S \coloneqq \id{S} \cross f$. 
	Then the functors
	\begin{equation*}
		f_S^\ast \colon \fromto{\Sh(S \cross Y;\E)}{\Sh(S \cross X;\E)} \andeq f_S^{\ast,\hyp} \colon \fromto{\Shhyp(S \cross Y;\E)}{\Shhyp(S \cross X;\E)} 
	\end{equation*}
	preserve locally $S$-constant and $S$-hyperconstant sheaves.
	Hence the assignments
	\begin{equation*}
		\goesto{Y}{\LC_S(S \cross Y;\E)} \andeq \goesto{Y}{\LChyp_S(S \cross Y;\E)}
	\end{equation*}
	define subfunctors of the functors
	\begin{equation*}
		\Sh(S \cross -;\E),\Shhyp(S \cross -;\E) \colon \fromto{\Top^{\op}}{\Catinfty} \period
	\end{equation*}
	Moreover, they are hypercomplete sheaves with respect to the open topology on $\Top$.
\end{observation}

\begin{observation}
	Let $ X $ be a topological space and $ g \colon \fromto{S}{T} $ a map of topological spaces.
	Write $g_X \coloneqq g \cross \id{X}$.
	Then the functors
	\begin{equation*} 
		g_X^\ast \colon \Sh(T \cross X;\E) \longrightarrow \Sh(S \cross X;\E) \andeq g_X^{\ast,\hyp} \colon \Shhyp(T \cross X; \E) \longrightarrow \Shhyp(S \cross X; \E) 
	\end{equation*}
	carry locally $T$-constant sheaves to locally $S$-constant sheaves and locally $T$-hyperconstant hypersheaves to locally $S$-hyperconstant hypersheaves.
	In particular, objects of \smash{$\LChyp_S(S \cross X; \E)$} can be seen as \emph{families} of objects in \smash{$\LChyp(X;\E)$} parametrized by the points of $S$.
\end{observation}

\begin{warning}
	Let $S$ and $X$ be topological spaces.
	If $V_\bullet$ is a hypercover of $S$, we obtain a commutative square
	\begin{equation*} 
	\begin{tikzcd}
		\LChyp_S(S \cross X; \E) \arrow[d, hooked] \arrow[r] & \lim_{[n] \in \DDelta} \LChyp_{V_n}( V_n \cross X ; \E) \arrow[d, hooked] \\
		\Shhyp(S \cross X; \E) \arrow[r] & \lim_{[n] \in \DDelta} \Shhyp( V_n \cross X ; \E ) \period
	\end{tikzcd} 
	\end{equation*}
	The vertical functors are fully faithful, and the bottom horizontal functor is an equivalence.
	It follows that the top horizontal functor is fully faithful as well.
	In general, there is no reason for it to be essentially surjective; nonetheless we will show that this is the case if $X$ is \wclwc (see \cref{cor:descent}).
\end{warning}


\section{Sheaves on locally weakly contractible topological spaces}\label{sec:hyperconstant}

Let $ S $ and $ X $ be \wclwc topological spaces and $ \cE $ a presentable \category.
The first goal of this section is to prove that if $ X $ \wclwc, then the pullback functor
\begin{equation*}
	\prSupperstarhyp \colon \Shhyp(S;\cE) \longrightarrow \Shhyp(S \cross X;\cE) 
\end{equation*}
is fully faithful, and to identify its essential image with the subcategory of locally $ S $-hyperconstant hypersheaves introduced in \cref{def:LC} (see \cref{thm:relative_full_faithfulness}).
After that, we explore some important consequences of this result.

In \cref{subsec:constant_hypersheaf}, we provide an alternative characterization of the pullback \smash{$ \prSupperstarhyp $}.
The alternative description guarantees that \smash{$ \prSupperstarhyp $} admits a left adjoint, which we refer to as the \textit{exceptional pushforward}.
\Cref{subsec:exceptional_pushforward} explores the basechange properties of the exceptional pushforward.
In \cref{subsec:full_faithfulness}, we use the alternative description of \smash{$ \prSupperstarhyp $} to prove \Cref{thm:relative_full_faithfulness}.
Finally, in \cref{subsec:foliated_hypersheaves} we reinterpret our \category \smash{$\LChyp_S(S \cross X;\E)$} in terms of \textit{foliated} hypersheaves.


\subsection{Formula for the hypersheaf pullback}\label{subsec:constant_hypersheaf}

Fix topological spaces $ S $ and $ X $, and a presentable \category $\E$.
Our first goal is to show that if $X$ is locally weakly contractible, then the functor \smash{$ \prSupperstarhyp $} admits a left adjoint. 
To do this, we provide an alternative description of \smash{$ \prSupperstarhyp $}.
This is easiest to describe when $ S = \pt $ and $ \E = \Spc $.
Write $ \Piinfty \colon \fromto{\Top}{\Spc} $ for the functor sending a topological space to its underlying homotopy type.
In this case, we show that the constant hypersheaf functor \smash{$ \fromto{\Spc}{\Shhyp(X)} $} is given by the assignment
\begin{equation*}
	E \mapsto [U \mapsto \Map_{\Spc}(\Piinfty(U),E)] \period
\end{equation*}
The general construction is just a relative version of this functor.
For the following constructions, recall the notations for posets of open subsets introduced in \Cref{nul:topological_space,eg:contractible_opens_are_a_basis}.


\begin{construction}\label{constr:Piinf}
	Consider the functor
	\begin{equation*} 
		\Piinfty(-/S) \colon \PSh(\Opencross(S \cross X)) \longrightarrow \Shhyp(S) 
	\end{equation*}
	left Kan extended from the functor $\Opencross(S \cross X) \to \Shhyp(S)$ sending $V \cross U$ to $V \tensor \Piinfty(U)$.
	This functor admits a right adjoint
	\begin{equation*}
		\Pi^{\infty}(-/S) \colon \Shhyp(S) \to \PSh(\Opencross(S \cross X))
	\end{equation*}
	given by the assignment
	\begin{equation*}
		G \mapsto [W \mapsto \Map_{\Shhyp(S)}(\Piinfty(W/S),G)] \period
	\end{equation*}
	By \cite[\HAappthm{Proposition}{A.3.2} \& \HAappthm{Lemma}{A.3.10}]{HA}, the functor $ \Pi_{\infty}(-/S)$ takes hypercover diagrams to colimits in $\Shhyp(S)$, and it therefore factors through
	\begin{equation*}
		\Shhyp(\Opencross(S \cross X)) \equivalent \Shhyp(S \cross X) \period
	\end{equation*}
	(The above equivalence follows from the fact that \smash{$ \Opencross(S \cross X) $} is a basis for the opens of $ S \cross X $; see \Cref{prop:hypersheaves_as_basis_RKE}.)
	Consequently, it follows that $\Pi^\infty(-/S)$ factors through $\Shhyp(S \cross X)$.
	We use the same notation for the resulting adjunction
	\begin{equation*}
		\adjto{\Piinfty(-/S)}{\Shhyp(S \cross X)}{\Shhyp(S)}{\Pi^{\infty}(-/S)} \period
	\end{equation*}

	Given a presentable \category $ \E $, write  $ \PiinftyE(-/S) \colon \fromto{\Shhyp(S \cross X;\E)}{\Shhyp(S;\E)} $ for the tensor product
	\begin{equation*} 
		\begin{tikzcd}[sep=8em]
			\Shhyp(S \cross X;\E) \equivalent \Shhyp(S \cross X) \tensor \E \arrow[r, "\Piinfty(-/S) \tensor \id{\E}"] & \Shhyp(S) \tensor \E \equivalent \Shhyp(S;\E) \period
		\end{tikzcd}
	\end{equation*}
	We write
	\begin{equation*}
		\PiEinfty(-/S) \colon \fromto{\Shhyp(S;\E)}{\Shhyp(S \cross X;\E)}
	\end{equation*}
	for the right adjoint of $ \PiinftyE(-/S) $.
	Concretely, $ \PiEinfty(-/S) $ is defined by sending $ G \in \Shhyp(S;\E)$ to the $\E$-valued hypersheaf
	\begin{equation*} 
		V \cross U \mapsto G(V)^{\Piinfty(U)} \semicolon 
	\end{equation*}
	here the exponential notation denotes the cotensoring of $\E$ over $ \Spc $.
\end{construction}

\begin{observation}\label{obs:base_change_Piinfty}
	Let $ X $ be a topological space and $ g \colon \fromto{S}{T} $ a map of topological spaces.
	Write $g_X \coloneqq g \cross \id{X}$.
	For $V \cross U \in \Opencross(T \cross X)$ there are canonical and functorial identifications
	\begin{align*}
		\gupperstarhyp( \Piinfty(V \cross U / T) ) &\simeq \gupperstarhyp( V \tensor \Piinfty(U) ) \\ 
		&\simeq \gupperstarhyp(V) \tensor \Piinfty(U) \\
		&\simeq \Piinfty( \gupperstarhyp_X(V \cross U) / S ) \period
	\end{align*}
	This implies that the diagram of left adjoints
	\begin{equation}\label{sq:Pi_infty} 
		\begin{tikzcd}
			\Shhyp(T \cross X) \arrow[r, "g_{X}^{\ast,\hyp}"] \arrow[d, "\Piinfty(-/T)"'] & \Shhyp(S \cross X) \arrow[d, "\Piinfty(-/S)"] \\
			\Shhyp(T) \arrow[r, "g^{\ast,\hyp}"'] & \Shhyp(S)
		\end{tikzcd} 
	\end{equation}
	is canonically commutative.
	Given a presentable \category $\E$, tensoring the square \eqref{sq:Pi_infty} with $\E$, we see that the same commutativity holds with coefficients in $\E$.
\end{observation}

We now compare the functor $ \PiEinfty(-/S) $ to the hypersheaf pullback \smash{$ \prSupperstarhyp $}.

\begin{construction}
	Fix $G \in \Shhyp(S; \E)$ and $V \in \Open(S)$.
	The unique map $\Piinfty(X) \to \pt$ induces a map
	\begin{equation*} 
		\begin{tikzcd}[sep=5em]
			G(V) \simeq G(V)^{\pt} \arrow[r] & G(V)^{\Piinfty(X)} \simeq \PiEinfty(G/S)(V \cross X) \simeq \prSlowerstar( \PiEinfty(G/S) )(V) \period
		\end{tikzcd} 
	\end{equation*}
	By adjunction, this corresponds to a map $\alpha_G \colon \prSupperstarhyp(G) \to \PiEinfty(G/S)$.
	These maps assemble together into a natural transformation
	\begin{equation*} 
		\alpha \colon \prSupperstarhyp \longrightarrow \PiEinfty(-/S) 
	\end{equation*}
	of functors $\Shhyp(S;\E) \to \Shhyp(S \cross X;\E)$.
\end{construction}

\begin{proposition}\label{prop:computing_hyperconstant_hypersheaf}
	Let $S$ and $X$ be topological spaces.
	Assume that $X$ is locally weakly contractible.
	Then the natural transformation
	\begin{equation*}
		\alpha \colon \prSupperstarhyp \to \PiEinfty(-/S)
	\end{equation*}
	is an equivalence.
	In particular, the functor \smash{$ \prSupperstarhyp $} is right adjoint to $ \PiinftyE(-/S) \colon \fromto{\Shhyp(S \cross X;\E)}{\Shhyp(S;\E)} $.
\end{proposition}

\begin{proof}
	First we treat the case where $\cE = \Spc$.
	Let
	\begin{equation*}
		j \colon \Openctrall(S \cross X)^{\op} \inclusion \Open(S \cross X)^{\op} \andeq i \colon \Shhyp(S \cross X) \inclusion \PSh(S \cross X)
	\end{equation*}
	denote the inclusions.
	Write $\unit \colon \pr_S\inv \to i \prSupperstarhyp$ for the unit.
	Write
	\begin{equation*} 
		\alphatilde \colon \pr_S\inv \longrightarrow i\Pi^\infty(-/S) 
	\end{equation*}
	for the composite of $i(\alpha)$ with $\unit$.
	Fix $F \in \Shhyp(S)$ and let $V \cross U \in \Openctrall(S \cross X)$.
	Unraveling the definitions shows that
	\begin{equation*} 
		\pr_S\inv(F)(V \cross U) \simeq F(V) \andeq i \Pi^\infty(F/S)(V \cross U) \simeq F(V)^{\Pi_\infty(U)} \period 
	\end{equation*}
	Moreover, the map $\widetilde{\alpha}$ is induced by the unique map $\Pi_\infty(U) \to \pt $.
	Since $U$ is weakly contractible, we deduce that for every \smash{$F \in \Shhyp(S)$}, the map $j^\ast(\alphatilde)$ is an equivalence.
	
	Since $\Pi^\infty(F/S)$ is a hypersheaf, it follows from \enumref{prop:hypersheaves_as_basis_RKE}{1} that $j^\ast( \pr_S\inv(F) )$ is a hypersheaf on \smash{$\Openctrall(S \cross X)$}.
	By \Cref{rem:hypersheafification} we can identify the unit \smash{$\pr_S\inv(F) \to i \prSupperstarhyp(F)$} with the unit
	\begin{equation*} 
		\pr_S\inv(F) \longrightarrow j_\ast j^\ast(\pr_S\inv(F)) \simeq j_\ast\jupperstar(i\Pi^\infty(F/S) ) \period 
	\end{equation*}
	Using \enumref{prop:hypersheaves_as_basis_RKE}{1} once more shows that the unit map $i\Pi^\infty(F/S) \to j_\ast\jupperstar(i\Pi^\infty(F/S) )$ is an equivalence, as desired.
	
	Now we treat the case where $\E$ is any presentable \category.
	Since we just showed that $\Pi^\infty(-/S)$ is equivalent to \smash{$ \prSupperstarhyp $}; it follows that $\Pi^\infty(-/S)$ commutes with colimits.
	The functoriality of tensor product of presentable \categories implies therefore that $\PiEinfty(-/S) \simeq \Pi^\infty(-/S) \tensor \id{\E}$.
	Since the same holds for the functor \smash{$\prSupperstarhyp$} and the map $\alpha$ respects such decomposition, the conclusion follows.
\end{proof}

\begin{remark}[(truncated coefficients)]\label{rem:local_weak_n-connectedness}
	Let $ n \geq 1 $ be an integer and let $ \E $ be a presentable $ n $-category.
	In this setting, to prove \Cref{prop:computing_hyperconstant_hypersheaf}, we only need to assume that $ X $ is \textit{locally weakly $ (n-1) $-connected} in the following sense: there is a basis of opens $ U \subset X $ such that $ \uppi_0(U) = \ast $ and all of the homotopy groups of $ U $ in degrees $ \leq n-1 $ vanish.
	In this case, in \Cref{constr:Piinf} we replace the underlying homotopy type $ \Piinfty(U) $ by the fundamental $ (n-1) $-groupoid $ \Pi_{n-1}(U) $.
	That is, we use the $ (n-1) $-truncation of $ \Piinfty(U) $.
	In particular, when $ \E = \Set $ and $ S = \ast $, the constant sheaf functor $ \fromto{\Set}{\Sh(X;\Set)} $ is given by sending a set $ E $ to the sheaf
	\begin{equation*} 
		\goesto{U}{\Map_{\Set}(\uppi_0(U),E)} \period
	\end{equation*}
	For $ n = 2 $ and $ S = \pt $, these results (essentially) recover results of Polesello--Waschkies \cite[\S\S2.1--2.2]{MR2155521}.

	All of the results in the rest of the paper can be formulated with coefficients in a presentable $ n $-category replacing assumptions of (local) weak contractibility with assumptions of (local) weak $ (n-1) $-connectedness.
	The proofs are exactly the same, replacing $ \Piinfty $ by $ \Pi_{n-1} $.
	Since we are most interested in \categories that are \textit{not} truncated, we will not explicitly highlight this generalization in the rest of the text.
\end{remark}


\subsection{The exceptional pushforward}\label{subsec:exceptional_pushforward}

Before moving on to the main result of this section, we need a brief digression about the exceptional pushforward whose existence is guaranteed by \cref{prop:computing_hyperconstant_hypersheaf}:

\begin{notation}
	Let $S$ and $X$ be topological spaces and assume that $X$ is locally weakly contractible.
	In light of \Cref{prop:computing_hyperconstant_hypersheaf}, we write
	\begin{align*}
		\prSlowersharphyp \colon \fromto{\Shhyp(S \cross X;\E)}{\Shhyp(S;\E)}
	\end{align*}
	for the left adjoint to \smash{$ \prSupperstarhyp $}.
	We refer to \smash{$\prSlowersharphyp$} as the \defn{exceptional pushforward}.
\end{notation}

\begin{corollary} \label{cor:projection_left_adjointable}
	Let $g \colon S \to T$ be a map of topological spaces, let $X$ be a locally weakly contractible topological space, and let $ \E $ be a presentable \category.
	Then the squares
	\begin{equation*}
		\begin{tikzcd}[row sep=2.5em, column sep=4em]
			\Shhyp(T; \E) \arrow[r, "g^{*,\hyp}"] \arrow[d, "\prSupperstarhyp"'] & \Shhyp(S; \E) \arrow[d, "\pr_T^{*,\hyp}"] \\
			\Shhyp(T \cross X ; \E) \arrow[r, "g_X^{*,\hyp}"'] & \Shhyp(S \cross X; \E)
		\end{tikzcd}
		\andeq 
		\begin{tikzcd}[column sep=4em, row sep=2.5em]
			\Shhyp(S \cross X;\E) \arrow[r, "{g_{X,\ast}}"] \arrow[d, "\prSlowerstar"'] & \Shhyp(T \cross X;\E) \arrow[d, "\prTlowerstar"] \\ 
			\Shhyp(S;\E) \arrow[r, "\glowerstar"'] & \Shhyp(T;\E) 
		\end{tikzcd} 
	\end{equation*}
	are vertically left adjointable.
	In particular, taking $T = \pt$ shows that if $F \in \Shhyp(S \cross X;\E)$ is a $S$-hyperconstant hypersheaf on $S \cross X$, then $\pr_{X,*}(F)$ is a hyperconstant hypersheaf on $X$.
\end{corollary}

\begin{proof}
	We have to prove that the exchange transformations
	\begin{equation*} 
		\prTlowersharphyp \circ g_X^{*,\hyp} \longrightarrow g^{*,\hyp} \circ \prSlowersharphyp \andeq \prTupperstarhyp \circ g_* \longrightarrow g_{X,*} \circ \pr_S^{*,\hyp}
	\end{equation*}
	are equivalences.
	The one on the right can be deduced from the one on the left by passing to right adjoints.
	By \Cref{prop:computing_hyperconstant_hypersheaf}, 
	\begin{equation*}
		\prSlowersharphyp \equivalent \Piinfty^\E(-/S) \andeq \prTlowersharphyp \equivalent \Piinfty^\E(-/T) \period
	\end{equation*}
	Hence the conclusion follows from \cref{obs:base_change_Piinfty}.
\end{proof}

\begin{remark}\label{rem:projection_left_adjointable}
	In the statement of \cref{cor:projection_left_adjointable}, one could ask whether the right-hand square is vertically \emph{right} adjointable (a question involving the regular pushforward instead of the exceptional one).
	However, there are two problems that prevent the resulting statement from being true: 
	\begin{enumerate}[label=\stlabel{rem:projection_left_adjointable}, ref=\arabic*]
		\item The topological space $ X $ is not assumed to be compact, but only locally weakly contractible.
		
		\item Even when $X$ is compact, we are working with \emph{hyper}sheaves, so the Proper Basechange Theorem \HTT{Corollary}{7.3.1.18} does not apply; see \HTT{Counterexample}{6.5.4.2.}.
	\end{enumerate}
\end{remark}

\begin{corollary} \label{cor:projection_locally_S_hyperconstant}
	Let $S$ and $X$ be topological spaces and assume that $X$ is locally weakly contractible.
	Then the pushforward functor \smash{$\pr_{X,*} \colon \Shhyp(S \cross X; \E) \longrightarrow \Shhyp(X; \E)$} restricts to a functor
	\begin{equation*} 
		\pr_{X,*} \colon \LChyp_S(S \cross X; \E) \longrightarrow \LChyp(X; \E) \period 
	\end{equation*}
\end{corollary}

\begin{proof}
	Since the formation of \smash{$\pr_{X,*}$} is compatible with restriction to an open subset of $X$, the question is local on $X$.
	Thus it is enough to check that if $F$ is a $S$-hyperconstant hypersheaf, then \smash{$\pr_{X,*}(F) \in \LChyp(X;\E)$}.
	This is guaranteed by \cref{cor:projection_left_adjointable}.
\end{proof}

We conclude this section by using \Cref{cor:projection_left_adjointable} to show that one can check that the unit morphism \smash{$ \unit_{F} \colon \fromto{F}{\prSupperstar\prSlowersharphyp(F)} $} is an equivalence locally.
In fact, we prove a slightly more general result that applies to the strata of a suitable stratification:

\begin{corollary}\label{cor:pullbackcounit}
	Let $ S $ and $ X $ be topological spaces and $ \{f_{\alpha} \colon \fromto{S_{\alpha}}{S}\}_{\alpha \in A} $ a collection of maps of topological spaces.
	Assume that $X$ is \wclwc and that the hypersheaf pullback functors
	\begin{equation*}
		\{(f_{\alpha} \cross \id{X})\upperstarhyp \colon \fromto{\Shhyp(S \cross X;\E)}{\Shhyp(S_{\alpha} \cross X;\E)}\}_{\alpha \in A}
	\end{equation*}
	are jointly conservative.
	Then the unit \smash{$ \fromto{F}{\prSupperstarhyp\prSlowersharphyp(F)} $} is an equivalence if and only if for each $ \alpha \in A $, the unit
	\begin{equation*}
		(f_{\alpha} \cross \id{X})\upperstarhyp(F) \to \prupperstarhyp_{S_{\alpha}}\pr_{S_{\alpha},\sharp}^{\hyp}(f_{\alpha} \cross \id{X})\upperstarhyp(F)
	\end{equation*}
	is an equivalence.
\end{corollary}

\begin{proof}
	\Cref{cor:projection_left_adjointable} implies that $(f_\alpha \cross \id{X})^{*,\hyp}$ takes the unit of the adjunction \smash{$\prSlowersharphyp \leftadjoint \prSupperstarhyp$} to the unit of the adjunction \smash{$\pr_{S_\alpha,\sharp}^{\hyp} \leftadjoint \pr_{S_\alpha}^{*,\hyp}$}.
	The conclusion follows.
\end{proof}


\subsection{Full faithfulness of the hypersheaf pullback}\label{subsec:full_faithfulness}

Now we prove \Cref{thm_intro:full_faithfulness}.

\begin{notation}\label{ntn:contractibility_total_category}
	We write $ \Env \colon \fromto{\Catinfty}{\Spc} $ for the left adjoint to the inclusion $ \Spc \subset \Catinfty $.
	Recall that for \acategory $ \cC $, the space $ \Env(\cC) $ can be computed as the colimit of the constant functor $ \cC \to \Spc$ at the terminal object $ \pt \in \Spc $ \cite[Corollary 2.10]{arXiv:2108.01924}.
\end{notation}

\begin{theorem}\label{thm:relative_full_faithfulness}
	Let $S$ and $X$ be topological spaces and assume that $X$ is \wclwc.
	Then:
	\begin{enumerate}[label=\stlabel{thm:relative_full_faithfulness}, ref=\arabic*]
		\item\label{thm:relative_full_faithfulness.1} The hypersheaf pullback \smash{$ \prSupperstarhyp \colon \Shhyp(S;\cE) \longrightarrow \Shhyp(S \cross X;\E) $} is fully faithful.
		
		\item\label{thm:relative_full_faithfulness.2} The essential image of \smash{$ \prSupperstarhyp $} consists of the \emph{locally} $S$-hyperconstant hypersheaves.
	\end{enumerate}
\end{theorem}

\begin{remark}
	Since the objects in the essential image of $\prSupperstarhyp$ are (by definition) globally $ S $-hyperconstant sheaves, we can rephrase \enumref{thm:relative_full_faithfulness}{2} as follows: if $X$ is \wclwc, then every locally $S$-hyperconstant sheaf is automatically globally $S$-hyperconstant.
\end{remark}

\begin{proof}
	For \enumref{thm:relative_full_faithfulness}{1}, note that since $\prSupperstarhyp $ is left adjoint to $ \prSlowerstar $, it suffices to provide a natural equivalence \smash{$\prSlowerstar \prSupperstarhyp \equivalent \id{} $} \cite[Lemma 3.3.1]{MR4246977}.
	Now note that since $ X $ is weakly contractible, applying \Cref{prop:computing_hyperconstant_hypersheaf} we see that for \smash{$ G \in \Shhyp(S;\E) $} and $V \in \Open(S)$ we have natural equivalences
	\begin{align*}
		\paren{\prSlowerstar\prSupperstarhyp(G)}(V) &\equivalent \paren{\PiEinfty(G/S)}(V \cross X) \\ 
		&\simeq G(V)^{\Piinfty(X)} \simeq G(V)^\pt \simeq G(V) \period
	\end{align*}
	
	Now we prove \enumref{thm:relative_full_faithfulness}{2}.
	Let $F \in \LChyp_S(S \cross X;\E)$.
	It suffices to prove that the counit
	\begin{equation*} 
		\counit \colon \prSupperstarhyp\prSlowerstar(F) \longrightarrow F 
	\end{equation*}
	is an equivalence.
	Let $\cB_F$ be the full subposet of $\Openctr(X)$ formed by those weakly contractible opens $U$ such that $\restrict{F}{S \cross U}$ is hyperconstant.
	Since $X$ is locally weakly contractible and $F$ is locally $S$-hyperconstant, the inclusion
	\begin{equation*} 
		\Open(S) \cross \cB_F \inclusion \Openctrall(S \cross X) \inclusion \Open(S \cross X) 
	\end{equation*}
	is a basis for $\Open(S \cross X)$.
	Since both the source and target of $ \counit $ are hypersheaves, \enumref{prop:hypersheaves_as_basis_RKE}{1} shows that it suffices to check that $\counit$ is an equivalence when restricted to $\Open(S) \cross \cB_F$.
	Fix $U \in \cB_F$ and write $q_U \colon S \cross U \to S$ for the projection; note that we have a natural identification
	\begin{equation*} 
		\restrict{\paren{\prSupperstarhyp\prSlowerstar(F)}\big}{S \cross U} \simeq q_U^{\ast,\hyp}\prSlowerstar(F) \period
	\end{equation*}
	Since $U$ is \wclwc, statement \enumref{thm:relative_full_faithfulness}{1} implies that the pushforward of \smash{$ q_U^{\ast,\hyp}\prSlowerstar(F) $} along $q_U$ canonically coincides with $\prSlowerstar(F)$.
	It follows that the counit transformation $\counit$ evaluated on $V \cross U \in \Open(S) \cross \cB_F$ is identified with the restriction morphism
	\begin{equation}\label{eq:equational_characterization}
		F(V \cross X) \longrightarrow F(V \cross U) \period
	\end{equation}
	Setting
	\begin{equation*} 
		F_V \coloneqq \prXlowerstar( \restrict{F}{V \cross X} ) \in \Shhyp(X;\E) \comma 
	\end{equation*}
	we are reduced to proving that for every fixed $V \in \Open(S)$ and every $U \in \Openctr(X)$, the restriction map
	\begin{equation*}
		F_V(X) \to F_V(U)
	\end{equation*}
	is an equivalence.
	\Cref{cor:projection_left_adjointable} implies that $F_V \in \LChyp(X;\E)$; we are therefore reduced to the case $S = \pt$.
	
	Let $j \colon \cB_F^{\op} \inclusion \Openctr(X)^{\op}$ denote the inclusion.
	\Cref{prop:hypersheaves_as_basis_RKE} guarantees that the unit transformation $F_V \to \jlowerstar \jupperstar(F_V)$ is an equivalence.
	It follows that for every $V \in \Open(S)$, the natural map
	\begin{equation*} 
		F_V(X) \longrightarrow \lim_{U \in \cB_F} F_V(U) 
	\end{equation*}
	is an equivalence.
	We claim that the functor $\jupperstar(F_V) = F_V \circ j$ inverts every morphism in $\cB_F$.
	To see this, let $i \colon W \inclusion U$ be a morphism in $\cB_F$.
	Since $U \in \cB_F$, there exists an object $E \in \E$ and an equivalence $\Gamma_{U}^{\ast,\hyp}(E) \simeq \restrict{F_V}{U}$.
	Since $\Gamma_W = \Gamma_U \circ i$, it follows that $\restrict{F_V}{W} \simeq \Gamma_W^{\ast,\hyp}(E)$.
	Consider the commutative triangle
	\begin{equation*} 
		\begin{tikzcd}[column sep = small]
			& E \arrow[dr] \arrow[dl] \\
			\Gamma_{U,\ast}\paren{ \Gamma_U^{\ast,\hyp}(E)} \arrow[rr] & & \Gamma_{W,\ast}\paren{\Gamma_W^{\ast,\hyp}(E)} \period
		\end{tikzcd} 
	\end{equation*}
	The bottom horizontal morphism is naturally identified with the restriction map $F_V(i) \colon F_V(U) \to F_V(W)$.
	On the other hand, since both $W$ and $U$ are \wclwc, \enumref{thm:relative_full_faithfulness}{1} implies that both the diagonal morphisms are equivalences.
	The $ 2 $-of-$ 3 $ property implies that $F_V(i)$ is an equivalence as well.
	
	Thus, the functor $\jupperstar(F_V)$ factors through $\Env(\cB_F)$.
	Observe that the functor $\Piinfty \colon \cB_F \to \Spc$ is equivalent to the constant functor sending every object of $\cB_F$ to $\pt \in \Spc$.
	It follows from \Cref{ntn:contractibility_total_category} that
	\begin{equation*} 
		\Env(\cB_F) \simeq \colim_{V \in \cB_F} \Piinfty(V) \period 
	\end{equation*}
	Van Kampen's Theorem identifies this colimit with $\Piinfty(X)$.
	Since $X$ is weakly contractible, we conclude that $\Env(\cB_F) \simeq \pt $, and therefore that $\jupperstar(F_V)$ is a constant functor.
	Finally, since $ \Env(\cB_F) $ is contractible, the restriction maps
	\begin{equation*} 
		F_V(X) \simeq \lim_{U \in \cB_F} F(U) \longrightarrow F(U) 
	\end{equation*}
	are equivalences for every $U \in \cB_F$.
	The conclusion follows.
\end{proof}

Taking $S = \pt$ we obtain:

\begin{corollary}\label{cor:absolute_full_faithfulness}
	Let $X$ be a \wclwc topological space and let $\E$ be a presentable \category.
	Then:
	\begin{enumerate}[label=\stlabel{cor:absolute_full_faithfulness}, ref=\arabic*]
		\item The constant hypersheaf functor \smash{$ \GammaXupperstarhyp \colon \cE \longrightarrow \Shhyp(X;\E) $} is fully faithful.
		
		\item The essential image of \smash{$ \GammaXupperstarhyp $} consists of the \emph{locally} hyperconstant hypersheaves.
	\end{enumerate}
\end{corollary}


\subsection{Foliated hypersheaves}\label{subsec:foliated_hypersheaves}

We end this section with an alternative description of locally $ S $-hyperconstant hypersheaves on $ S \cross X $.
The idea is that in order to check local $ S $-hyperconstancy, it suffices to check hyperconstancy on the `leaves' $ \{s\} \cross X $.

\begin{definition}\label{def:foliated}
	Let $S$ and $X$ be topological spaces and assume that $X$ is \wclwc.
	Let $ \E $ be a presentable \category.
	A hypersheaf \smash{$ F \in \Shhyp(S \cross X;\E) $} is \defn{foliated} if for each $ s \in S $, the restriction \smash{$ \restrict{F}{\{s\} \cross X}^{\hyp} $} is a hyperconstant hypersheaf.
\end{definition}

\begin{example}\label{ex:pullback_is_foliated}
	Given a hypersheaf $ G \in \Shhyp(S;\E) $, the pullback $ \prSupperstarhyp(G) $ is foliated.
\end{example}

The following generalizes \HAa{Proposition}{A.2.5} from $ X = \RR $ to any \wclwc topological space.

\begin{proposition}\label{prop:classification_of_foliated}
	Let $S$ and $X$ be topological spaces and assume that $X$ is \wclwc.
	Let $ \E $ be a \emph{compactly generated} \category.
	For \smash{$F \in \Shhyp(S \cross X;\cE)$}, the following statements are equivalent:
	\begin{enumerate}[label=\stlabel{prop:classification_of_foliated}, ref=\arabic*]
		\item\label{prop:classification_of_foliated.1} The hypersheaf $ F $ is in the essential image of \smash{$ \prSupperstarhyp \colon \Shhyp(S;\cE) \inclusion \Shhyp(S \cross X;\E) $}.

		\item\label{prop:classification_of_foliated.2} The hypersheaf $ F $ is foliated.
	\end{enumerate}
\end{proposition}

\begin{proof}
	The implication \enumref{prop:classification_of_foliated}{1} $ \Rightarrow $ \enumref{prop:classification_of_foliated}{2} is the content of \Cref{ex:pullback_is_foliated}.

	To see that \enumref{prop:classification_of_foliated}{2} $ \Rightarrow $ \enumref{prop:classification_of_foliated}{1}, we need to show that if $ F $ is foliated, then the unit \smash{$ \unit_F \colon \fromto{F}{\prSupperstarhyp\prSlowersharphyp(F)} $} is an equivalence.
	Notice that since $ \E $ is compactly generated, the restriction functors
	\begin{equation*}
		\set{\restrict{(-)}{\{s\} \cross X}^{\hyp} \colon \fromto{\Shhyp(S \cross X;\E)}{\Shhyp(\{s\} \cross X;\E)} }_{s \in S}
	\end{equation*}
	are jointly conservative (\Cref{rec:stalksconservativeCGhyp}).
	Applying \Cref{cor:pullbackcounit}, we see that to prove that $ \unit_F $ is an equivalence, it suffices to show that for each $ s \in S $, the unit
	\begin{equation*}
		\begin{tikzcd}[sep=2.5em]
			\restrict{F}{\{s\} \cross X}^{\hyp} \arrow[r] & \prsupperstarhyp\prslowersharphyp\paren{\restrict{F}{\{s\} \cross X}^{\hyp}} \equivalent \GammaXupperstarhyp\GammaXlowersharphyp\paren{\restrict{F}{\{s\} \cross X}^{\hyp}}
 		\end{tikzcd}
	\end{equation*} 
	is an equivalence.
	The claim now follows from the assumption that \smash{$ \restrict{F}{\{s\} \cross X}^{\hyp} $} is hyperconstant combined with \Cref{cor:absolute_full_faithfulness}.
\end{proof}


\section{Consequences of the full faithfulness of the hypersheaf pullback}\label{sec:consequences_of_full_faithfulness}

In this section, we explore some immediate consequences of \Cref{thm:relative_full_faithfulness}.
We begin in \cref{subsec:structural_results} by deducing several structural results concerning the \category \smash{$\LChyp_S(S \cross X;\E)$}.
In \cref{subsec:monodromy} we deduce a monodromy equivalence and prove a Künneth formula for locally hyperconstant hypersheaves.
In \cref{subsec:HIforLChypersheaves} we establish the first homotopy-invariance result of the paper, and in \cref{subsec:behavioronLC} we analyze the behavior of the exceptional pushforward on locally hyperconstant hypersheaves.
Finally, in \cref{subsec:cohomology_comparison} we obtain a general comparison result for sheaf and singular cohomology for locally weakly contractible spaces.


\subsection{Structural results for locally hyperconstant hypersheaves}\label{subsec:structural_results}

\Cref{thm:relative_full_faithfulness} can be used to prove that the \category \smash{$\LChyp_S(S \cross X;\E)$} enjoys many nice properties.
We start with the following recognition criterion:

\begin{proposition}\label{prop:strong_local_hyperconstancy}
	Let $S$ and $X$ be topological spaces and let $\cE$ be a presentable \category.
	Assume that $X$ is locally weakly contractible.
	For \smash{$F \in \Shhyp(S \cross X;\cE)$}, the following statements are equivalent:
	\begin{enumerate}[label=\stlabel{prop:strong_local_hyperconstancy}, ref=\arabic*]
		\item\label{prop:strong_local_hyperconstancy.1} The sheaf $F$ is locally $S$-hyperconstant.
		
		\item\label{prop:strong_local_hyperconstancy.2} For every pair of weakly contractible open subsets $V \subset U$ of $X$ and every open subset $W \subset S$, the restriction map $ F(W \cross U) \to F(W \cross V) $ is an equivalence in $\cE$.
	\end{enumerate}
\end{proposition}

\begin{proof}
	We first prove that \enumref{prop:strong_local_hyperconstancy}{1} implies \enumref{prop:strong_local_hyperconstancy}{2}.
	Write
	\begin{equation*} 
		q_V \colon W \cross V \longrightarrow W \andeq q_U \colon W \cross U \longrightarrow W 
	\end{equation*}
	for the projections.
	Since $U$ is weakly contractible, \cref{thm:relative_full_faithfulness} implies that $F |_{W\cross U}$ is $W$-hyperconstant.
	We can then choose a hypersheaf $G \in \Shhyp(W;\cE)$ and an equivalence
	\begin{equation*}
		\restrict{F}{W\cross U} \simeq q_U^{*,\hyp}(G)
	\end{equation*}
	It follows that $\restrict{F}{W\cross V} \simeq q_V^{*,\hyp}(G)$.
	Now consider the commutative triangle
	\begin{equation*} 
		\begin{tikzcd}[column sep = small]
			& G \arrow[dr] \arrow[dl] \\
			q_{U,*} q_U^{*,\hyp}(G) \arrow[rr] & & q_{V,*} q_V^{*,\hyp}(G) \period
		\end{tikzcd} 
	\end{equation*}
	Since $U$ and $V$ are weakly contractible, the full faithfulness part of \cref{thm:relative_full_faithfulness} implies that the diagonal maps are equivalences.
	Thus the horizontal map is an equivalence.
	To conclude, note that, unraveling the definitions, this horizontal map coincides with the restriction map $F(W \cross U) \to F(W \cross V)$.
	
	We now prove that \enumref{prop:strong_local_hyperconstancy}{2} implies \enumref{prop:strong_local_hyperconstancy}{1}.
	By choosing an open cover of $ X $ by \wclwc opens, we are reduced to the case that $ X $ is \wclwc.
	Let $F$ be a hypersheaf satisfying assumption \enumref{prop:strong_local_hyperconstancy}{2}.
	Since $ X $ is \wclwc, is enough to prove that the counit
	\begin{equation*} 
		\counit_F \colon \prSupperstarhyp\prSlowerstar(F) \longrightarrow F  
	\end{equation*}
	is an equivalence.
	In the first segment of the proof of \enumref{thm:relative_full_faithfulness}{2} we proved that this is the same as saying that the for every $U \in \Openctr(X)$ and $W \in \Open(S)$, the restriction map $F(W \cross X) \to F(W \cross U)$ is an equivalence.
	Since $X$ and $U$ are weakly contractible, this is guaranteed by our hypothesis.
\end{proof}

\begin{corollary}\label{cor:locally_constant_closed_under_limits_colimits}
	Let $S$ and $X$ be topological spaces, and assume that $X$ is locally weakly contractible.
	Then the full subcategory
	\begin{equation*}
		\LChyp_S(S \cross X; \E) \subset \Shhyp(S \cross X; \E)
	\end{equation*}
	is closed under limits and colimits.
\end{corollary}

\begin{remark}
	\Cref{cor:locally_constant_closed_under_limits_colimits} implies that \smash{$\LChyp_S(S \cross X; \cE)$} is both a reflective and coreflective subcategory of \smash{$\Shhyp(S \cross X; \E)$}.
\end{remark}

\begin{proof}[Proof of \Cref{cor:locally_constant_closed_under_limits_colimits}]
	Let $A$ be a small \category and let $F_{\bullet} \colon A \to \LChyp_S(S \cross X;\cE)$ be a diagram.
	First we treat the case of limits.
	By \cref{prop:strong_local_hyperconstancy}, it is enough to prove that for every $V \subset U$ in \smash{$\Openctr(X)$}, and every $W \in \Open(S)$, the restriction map
	\begin{equation*}
		\lim_{\alpha \in A} F_\alpha(W \cross U) \longrightarrow \lim_{\alpha \in A} F_\alpha(W \cross V) 
	\end{equation*}
	is an equivalence.
	Since limits in $\Shhyp(X;\cE)$ are computed objectwise, the above map is the limit of the individual restriction maps $F_\alpha(W \cross U) \to F_\alpha(W \cross V)$.
	Since each $F_\alpha$ is locally $ S $-hyperconstant, \cref{prop:strong_local_hyperconstancy} implies that all these maps are equivalences.
	Thus, the same goes for their limit.
	
	For the case of colimits, we have to check that the colimit $ \colim_{\alpha \in A} F_{\alpha} $ computed in \smash{$\Shhyp(S \cross X; \E)$} is locally $S$-hyperconstant.
	The question is local on $X$, and we can assume that $X$ is weakly contractible.
	In this case, \cref{thm:relative_full_faithfulness} shows the functor \smash{$\prSupperstarhyp$} is fully faithful; thus there exists a diagram
	\begin{equation*}
		F_{\bullet}' \colon A \to \Shhyp(S; \E)
	\end{equation*}
	and an equivalence $F_{\bullet} \simeq \prSupperstarhyp \circ F_{\bullet}'$.
	The fact that $\prSupperstarhyp$ commutes with colimits completes the proof.
\end{proof}

\begin{corollary} \label{cor:descent}
	Let $S$ and $X$ be topological spaces and let $\E$ be a presentable \category.
	Assume that $X$ is \wclwc.
	Then for every hypercover $V_\bullet$ of $S$, the natural functor
	\begin{equation*} 
		\LChyp_S(S \cross X; \E) \longrightarrow \lim_{[n] \in \DDelta} \LChyp_{V_n}(V_n \cross X ; \E) 
	\end{equation*}
	is an equivalence.
\end{corollary}

\begin{proof}
	Consider the commutative square
	\begin{equation*} 
		\begin{tikzcd}
			\Shhyp(S;\E) \arrow[d, "\prSupperstarhyp"']  \arrow[r] & \lim_{[n] \in \DDelta} \Shhyp(V_n;\E) \arrow[d, "\pr_{V_\bullet}^{*,\hyp}"] \\
			\LChyp_S(S \cross X; \E) \arrow[r] & \lim_{[n] \in \DDelta} \LChyp_{V_n}(V_n \cross X; \E) \period
		\end{tikzcd} 
	\end{equation*}
	Since $ \Shhyp(-;\E) $ satisfies hyperdescent, the top horizontal functor is an equivalence.
	\Cref{thm:relative_full_faithfulness} implies that both vertical functors are equivalences.
	The conclusion follows.
\end{proof}


\subsection{Monodromy equivalence and Künneth formula}\label{subsec:monodromy}

Let $X$ be a topological space.
There is a natural map from the underlying homotopy type $\Piinfty(X)$ of $X$ to the shape of the \topos \smash{$ \Shhyp(X) $}.
However, this map is typically \textit{not} an equivalence.
Our work in \cref{sec:hyperconstant} implies that these invariants agree when $ X $ is locally weakly contractible:

\begin{corollary}\label{cor:Shape_of_ShhypX}
	Let $X$ be a locally weakly contractible topological space.
	Then the \topos $\Shhyp(X)$ is locally of constant shape, and its shape coincides with $\Piinfty(X)$.
\end{corollary}

\begin{proof}
	This is a direct consequence of \Cref{prop:computing_hyperconstant_hypersheaf} and \cite[\HAappthm{Proposition}{A.1.8} \& \HAappthm{Remark}{A.1.10}]{HA}.
\end{proof}

\begin{notation}
	Write $ \Toplwc \subset \Top $ for the full subcategory spanned by the locally weakly contractible topological spaces.
\end{notation}

\begin{corollary}[(monodromy equivalence)]\label{cor:monodromy}
	Let $X$ be a locally weakly contractible topological space.
	Then the functor
	\begin{equation}\label{eq:monodromy1}
		\Piinfty \colon \LChyp(X) \longrightarrow \Spc_{/\Piinfty(X)} 
	\end{equation}
	is an equivalence.
\end{corollary}

\begin{proof}
	\Cref{prop:computing_hyperconstant_hypersheaf} shows that \smash{$ \GammaXupperstarhyp $} is right adjoint to the functor \smash{$ \Piinfty \colon \fromto{\Shhyp(X)}{\Spc} $}.
	The conclusion follows then from \HAa{Theorem}{A.1.15}
\end{proof}

\begin{observation}\label{obs:monodromy}
	Unraveling the proof of \HAa{Theorem}{A.1.15}, we see that the inverse to \eqref{eq:monodromy1} is given by sending a map $ \fromto{K}{\Piinfty(X)} $ to the sheaf
	\begin{equation*}
		U \mapsto \Map_{/\Piinfty(X)}(\Piinfty(U),K) \period
	\end{equation*}
	Straightening/unstraightening puts the monodromy equivalence \eqref{eq:monodromy1} into a more familiar form:
	\begin{equation}\label{eq:monodromy2}
		\LChyp(X) \equivalent \Fun(\Piinfty(X),\Spc) \period
	\end{equation}
	Moreover, the equivalence \eqref{eq:monodromy2} refines to an equivalence of functors $ \fromto{\Toplwcop}{\Cat_{\infty}} $.
	In particular, the functor \smash{$ \LChyp \colon \fromto{\Toplwcop}{\Cat_{\infty}} $} inverts \textit{weak} homotopy equivalences between locally weakly contractible topological spaces.
\end{observation}

\begin{observation}\label{obs:general_monodromy}
	Let $ \E $ be a presentable \category.
	Since restriction of sheaves to an open subset is both a left and a right adjoint, the equivalence
	\begin{align*}
		\Shhyp(X) \tensor \E &\equivalence \Shhyp(X;\E) \\ 
		\intertext{restricts to an equivalence}
		\LChyp(X) \tensor \E &\equivalence \LChyp(X;\E) \period
	\end{align*}
	Thus tensoring \eqref{eq:monodromy2} with $ \E $ provides a monodromy equivalence
	\begin{equation}\label{eq:monodromy3}
		\LChyp(X;\E) \equivalent \Fun(\Piinfty(X),\E)
	\end{equation}
	for $ \E $-valued locally hyperconstant hypersheaves.
	Also note that the functoriality of the equivalence \eqref{eq:monodromy3} implies that given $ L \in \LChyp(X;\E) $, the associated functor $ \fromto{\Piinfty(X)}{\E} $ carries $ x \in X $ to the stalk $ \xupperstar L \in \E $.
\end{observation}

\begin{remark}[(the classical monodromy equivalence)]
	Write $ \Pi_1(X) $ for the fundamental groupoid of $ X $.
	Since $ \Pi_1(X) $ is the homotopy $ 1 $-category of $ \Piinfty(X) $,  if $ \E $ is a presentable $ 1 $-category, then
	\begin{equation*}
		\LChyp(X;\E) = \LC(X;\E) \andeq \Fun(\Piinfty(X),\E) \equivalent \Fun(\Pi_1(X),\E) \period
	\end{equation*}
	In particular, \Cref{obs:general_monodromy} recovers the classical monodromy equivalence for locally weakly contractible topological spaces.

	In the classical monodromy equivalence, only local $ 1 $-connectedness is needed, so this seems to use stronger hypotheses than the classical result.
	However, the truncated variants of our results (see \Cref{rem:local_weak_n-connectedness}) recover and generalize the classical monodromy equivalence.
	Let $ n \geq 1 $ and let $ \E $ be a presentable $ n $-category.
	If $ X $ is a locally weakly $ (n-1) $-connected topological space, then the constant sheaf functor $ \fromto{\E}{\Sh(X;\E)} $ admits a left adjoint.
	In particular, the \topos \smash{$ \Shhyp(X) $} is \textit{locally $ (n-1) $-connected} in the sense of \cite[Definition 3.2]{MR3763287}.
	Write \smash{$ \Pi_{n}\Shhyp(X) $} for the $ n $-truncation of the shape of the \topos \smash{$ \Shhyp(X) $}.
	Applying \cite[Theorem 3.13]{MR3763287} provides a monodromy equivalence
	\begin{equation}\label{eq:monodromy_truncated_shape}
		\LC(X;\E) \equivalent \Fun(\Pi_{n}\Shhyp(X),\E) \period
	\end{equation}

	If $ X $ is locally weakly $ n $-connected, then the formula for the constant sheaf functor \smash{$ \fromto{\Spc_{\leq n}}{\Sh(X;\Spc_{\leq n})} $} provided by \Cref{prop:computing_hyperconstant_hypersheaf} shows that \smash{$ \Pi_{n}\Shhyp(X) $} coincides with the fundamental $ n $-groupoid $ \Pi_n(X) $ of $ X $.
	(This is the truncated variant of \Cref{cor:Shape_of_ShhypX}.)
	Hence \eqref{eq:monodromy_truncated_shape} becomes an equivalence 
	\begin{equation*}
		\LC(X;\E) \equivalent \Fun(\Pi_{n}(X),\E) \period
	\end{equation*}
	Setting $ n = 1 $ we obtain a generalization of the classical monodromy equivalence to locally \textit{weakly} $ 1 $-connected topological spaces.
\end{remark}

We conclude this subsection with a \textit{categorical Künneth formula} for locally hyperconstant hypersheaves. 
Given topological spaces $ X $ and $ Y $, note that the functors
\begin{align*}
	\Sh(X) \cross \Sh(Y) \to \Sh(X \cross Y) &\andeq  \Shhyp(X) \cross \Shhyp(Y) \to \Shhyp(X \cross Y) \\
	\shortintertext{given by}
	(F,G) \mapsto \prupperstar_X(F) \cross \prupperstar_Y(G)  &\andeq \qquad(F,G) \mapsto \prupperstarhyp_X(F) \cross \prupperstarhyp_Y(G)
\end{align*}
preserve colimits separately in each variable.
Since the coproduct in the \category of \topoi and left exact left adjoints is given by the tensor product of presentable \categories \cites[\HAthm{Example}{4.8.1.19}]{HA}[Theorem 2.15]{arXiv:1802.10425}, these functors induce left exact colimit-preserving functors
\begin{equation*}
	\Sh(X) \tensor \Sh(Y) \to \Sh(X \cross Y) \andeq  \Shhyp(X) \tensor \Shhyp(Y) \to \Shhyp(X \cross Y) \period
\end{equation*}
In general, neither of these functors is an equivalence.%
\footnote{If $ X $ or $ Y $ is locally compact Hausdorff, then the functor $ \Sh(X) \tensor \Sh(Y) \to \Sh(X \cross Y) $ is an equivalence \HTT{Proposition}{7.3.1.11}.}
Nonetheless, locally hyperconstant hypersheaves on $X \cross Y$ \textit{do} decompose as a tensor product:

\begin{corollary}[(Künneth formula)] \label{cor:Kunneth}
	Let $ X $ and $ Y $ be a locally weakly contractible topological spaces.
	The natural functor \smash{$ \LChyp(X) \cross \LChyp(Y) \to \LChyp(X \cross Y) $} induces an equivalence of \categories 
	\begin{equation*}
		\LChyp(X) \tensor \LChyp(Y) \equivalence \LChyp(X \cross Y) \period
	\end{equation*}
\end{corollary}

\begin{proof}
	Since $ \Piinfty $ preserves finite products and $ \Fun(-,\Spc) $ carries products of \categories to tensor products of presentable \categories, the conclusion follows from the monodromy equivalence \eqref{eq:monodromy2}.
\end{proof}


\subsection{Homotopy-invariance for locally hyperconstant hypersheaves}\label{subsec:HIforLChypersheaves}

Our next goal is to use \cref{thm:relative_full_faithfulness} to show that, for every presentable \category $\E$, the functor
\begin{equation*} 
	\LChyp(-;\E) \colon \Top^{\op} \longrightarrow \Cat_\infty
\end{equation*}
is strongly homotopy invariant in the sense of \Cref{def:homotopy_invariant}.
We need the following two preliminary results:

\begin{lemma} \label{lem:hyperconstant_implies_S_hyperconstant}
	Let $S$ and $X$ be topological spaces, and assume that $X$ is \wclwc.
	Then 
	\begin{equation*} 
		\LChyp(S \cross X; \E) \subset \LChyp_S(S \cross X; \E) \period 
	\end{equation*}
\end{lemma}

\begin{proof}
	Applying \cref{cor:descent}, we see that for every hypercover $V_\bullet$ of $S$, the natural functor
	\begin{equation*} 
		\LChyp_S(S \cross X; \E) \longrightarrow \lim_{[m] \in \DDelta} \LChyp_{V_m}(V_m \cross X ; \E) 
	\end{equation*}
	is an equivalence.
	On the other hand, for every hypercover $U_{\bullet}$ of $X$ where each $ U_n $ is \wclwc, the natural functor
	\begin{equation*} 
		\LChyp_{V_m}( V_m \cross X ; \E ) \longrightarrow \lim_{[n] \in \DDelta} \LChyp_{V_m}( V_m \cross U_n ) 
	\end{equation*}
	is an equivalence.
	Thus, a hypersheaf $F \in \Shhyp(S \cross X; \E)$ belongs to the full subcategory $\LChyp_S(S \cross X; \E)$ if and only if we can find a hypercover $V_\bullet \cross U_\bullet$ of $S \cross X$ such that, for every $([n], [m]) \in \DDelta \cross \DDelta$, the restriction $F |_{V_m \cross U_n}$ belongs to \smash{$\LChyp_{V_m}(V_m \cross U_n ; \E)$}.
	If $F \in \LChyp(S \cross X; \E)$, there exists a hypercover $V_\bullet \cross U_\bullet$ such that, for every $([n],[m]) \in \DDelta \cross \DDelta$, the restriction $\restrict{F}{V_m \cross U_n}$ is hyperconstant, hence $V_m$-hyperconstant. 
	The conclusion follows.
\end{proof}

\begin{lemma} \label{lem:exceptional_pushforward_locally_hyperconstant}
	Let $S$ and $X$ be topological spaces, and assume that $X$ is \wclwc.
	Then the pushforward
	\begin{equation*}
		\prSlowerstar \colon \Shhyp(S \cross X; \E) \to \Shhyp(S; \E)
	\end{equation*}
	preserves locally hyperconstant hypersheaves.
\end{lemma}

\begin{proof}
	Let $F \in \LChyp(S \cross X; \E)$.
	By \cref{lem:hyperconstant_implies_S_hyperconstant}, we know that \smash{$F \in \LChyp_S(S \cross X; \E)$}.
	Thus there exists a hypersheaf $G $ on $ S $ and an equivalence \smash{$F \simeq \prSupperstarhyp(G)$}.
	Since \smash{$\prSupperstarhyp$} is fully faithful (\cref{thm:relative_full_faithfulness}), the unit defines an equivalence $\equivto{G}{\prSlowerstar(F)}$.
	Hence our goal is to show that $G$ is locally hyperconstant.
	
	Since $F \in \LChyp(S \cross X; \E)$, there exists an open cover $\{V_\alpha \cross U_\alpha\}_{\alpha \in A}$ of $S \cross X$ such that for each $\alpha \in A$, the hypersheaf $F |_{V_\alpha \cross U_\alpha}$ is hyperconstant.
	Since $X$ is locally weakly contractible, we can furthermore assume that every $U_\alpha$ is weakly contractible.
	Write $q_\alpha \colon V_\alpha \cross U_\alpha \to V_\alpha$ for the projection.
	Since \smash{$F \simeq \prSupperstarhyp(G)$}, we see that
	\begin{equation*}
		\restrict{F}{V_\alpha \cross U_\alpha} \simeq q_\alpha^{*,\hyp}(\restrict{G}{V_\alpha}) \period
	\end{equation*}
	Since $U_\alpha$ is weakly contractible, using \cref{thm:relative_full_faithfulness} again we see that the unit
	\begin{equation*} 
		\restrict{G}{V_{\alpha}} \longrightarrow q_{\alpha,*}( \restrict{F}{V_\alpha \cross U_\alpha} ) 
	\end{equation*}
	is an equivalence.
	We can therefore replace $S$ and $X$ by $U_\alpha$ and $V_\alpha$, respectively.
	Equivalently, we can assume from the beginning that $F$ is globally hyperconstant.
	We can therefore write \smash{$F \simeq \Gamma_{S \cross X}^{*,\hyp}(E)$}, for some object $E \in \E$.
	In this case, we obtain equivalences
	\begin{equation*} 
		\prSupperstarhyp(G) \simeq F \simeq \Gamma_{S \cross X}^{*,\hyp}(E) \simeq \prSupperstarhyp( \Gamma_S^{*,\hyp}(E)) \period
	\end{equation*}
	Applying \cref{thm:relative_full_faithfulness} once more, we deduce that
	\begin{equation*} 
		G \simeq \prSlowerstar(F) \simeq \Gamma_S^{*,\hyp}(E) \period \qedhere
	\end{equation*}
\end{proof}

\begin{theorem}\label{thm:homotopy_invariance_hyperconstant}
	Let $S$ and $X$ be topological spaces, and assume that $X$ is \wclwc.
	Then the functors
	\begin{equation*}
		\adjto{\prSupperstarhyp}{\LChyp(S; \E)}{\LChyp(S \cross X; \E)}{\prSlowerstar}
	\end{equation*}
	are inverse equivalences of \categories.
	In particular, the functor $ \LChyp(-;\E) \colon \fromto{\Top^{\op}}{\Catinfty} $ is strongly homotopy-invariant
\end{theorem}

\begin{remark}
	When $S$ is itself locally weakly contractible, \Cref{thm:homotopy_invariance_hyperconstant} is a consequence of the monodromy equivalence (see \Cref{obs:monodromy}).
	The strength of \cref{thm:relative_full_faithfulness} is that we have no assumptions on $ S $.
\end{remark}

\begin{proof}[Proof of \Cref{thm:homotopy_invariance_hyperconstant}]
	In virtue of \cref{lem:hyperconstant_implies_S_hyperconstant}, we can consider the following commutative square:
	\begin{equation*} 
		\begin{tikzcd}
			\LChyp(S;\E) \arrow[d, "\prSupperstarhyp"'] \arrow[r, hooked] & \Shhyp(S;\E) \arrow[d, "\prSupperstarhyp"] \\
			\LChyp(S \cross X; \E) \arrow[r, hooked] & \LChyp_S(S \cross X; \E) \period
		\end{tikzcd} 
	\end{equation*}
	\cref{thm:relative_full_faithfulness} implies that the right vertical functor is an equivalence.
	Since the horizontal functors are fully faithful, \cref{lem:exceptional_pushforward_locally_hyperconstant} implies that this square is vertically right adjointable.
	The conclusion follows.
\end{proof}

 Although not needed in what follows, we remind the reader that the $2$-of-$6$ property of equivalences implies that a functor \smash{$ L \colon \fromto{\Top}{\Catinfty} $} is homotopy-invariant in the sense of \Cref{def:homotopy_invariant} if and only if $ L $ inverts \textit{all} homotopy equivalences of topological spaces:
%

\begin{lemma}\label{lem:LC_homotopy_invariance_reformulation}
	The following are equivalent for a functor \smash{$ L \colon \fromto{\Top}{\Catinfty} $}:
	\begin{enumerate}[label=\stlabel{lem:LC_homotopy_invariance_reformulation}, ref=\arabic*]
			\item The functor $ L $ is homotopy-invariant.
			
			\item For each homotopy equivalence of topological spaces $ f \colon \fromto{T}{S} $, the functor $ L(f) \colon \fromto{L(S)}{L(T)} $ is an equivalence of \categories.
		\end{enumerate}
\end{lemma}



\subsection{Exceptional pushforward on locally hyperconstant hypersheaves}\label{subsec:behavioronLC}

We now prove that the exceptional pushforward preserves locally hyperconstant hypersheaves.
We start with the following observations:

\begin{observation}
	Let $X$ be a topological space and let $j \colon U \to X$ be a local homeomorphism.
	Then $j\inv \colon \PSh(X;\E) \to \PSh(U;\E)$ preserves (hyper)sheaves.
	Furthermore, the functor
	\begin{equation*} 
		j\inv \colon \Shhyp(X;\E) \longrightarrow \Shhyp(U;\E) 
	\end{equation*}
	commutes with arbitrary limits, hence admits a left adjoint \smash{$j_\sharp^{\hyp}$}.
	Observe that if $j$ is an open immersion, then \smash{$j_\sharp^{\hyp}$} coincides with the hypersheafification of the usual extension by zero.
\end{observation}

\begin{observation}
	Let $X$ be a topological space and let $U_\bullet$ be a hypercover of $X$.
	For every $[n] \in \DDelta$, denote by $j_n \colon U_n \to X$ the canonical morphism.
	Hyperdescent implies that the natural functor
	\begin{equation*} 
		j_\bullet^* \colon \Shhyp(X; \E) \longrightarrow \lim_{[n] \in \DDelta} \Shhyp(U_n; \E)
	\end{equation*}
	is an equivalence.
	In particular, $ j_\bullet^* $ admits a left adjoint, that we denote \smash{$j_{\bullet,\sharp}^{\hyp}$}.
	Using \cite[\S 8.2]{MR3545934}, the left adjoint \smash{$j_{\bullet,\sharp}^{\hyp}$} can be described as the functor sending a descent datum $\{F_n\}_{n \geq 0}$ to
	\begin{equation*} 
		j_{\bullet,\sharp}^{\hyp} \paren{\{F_n\}_{n \geq 0}} \simeq \colim_{[n] \in \Deltaop} j_{n,\sharp}^{\hyp}(F_n) \period 
	\end{equation*}
	In particular, for every hypersheaf $F \in \Shhyp(X;\E)$, there is a natural equivalence
	\begin{equation}\label{eq:exceptional_descent}
		F \simeq \colim_{[n] \in \Deltaop} j_{n,\sharp}^{\hyp}(j_n^*(F)) \period
	\end{equation}
\end{observation}

\begin{notation}\label{nul:generic_assumptions}
	For the remainder of this section, we fix topological spaces $S$ and $X$, as well as a presentable \category $\E$.
	Furthermore, we assume that $X$ is locally weakly contractible.
\end{notation}

\begin{lemma} \label{lem:exceptional_descent}
	Let $U_\bullet$ be a hypercover of $X$.
	For every $[n] \in \DDelta$, let $j_n \colon U_n \to X$ be the canonical morphism and set $p_n \coloneqq \pr_S \circ (\id{S} \cross j_n)$.
	Then for every \smash{$F \in \Shhyp(S \cross X;\E)$}, one has a natural equivalence
	\begin{equation*} 
		\prSlowersharphyp(F) \simeq \colim_{[n] \in \DDelta} p_{n,\sharp}^{\hyp} (\id{S} \cross j_n)^*(F) \period 
	\end{equation*}
\end{lemma}

\begin{proof}
	Since \smash{$ \prSlowersharphyp $} preserves colimits, the claim follows from applying \smash{$ \prSlowersharphyp $} to the equivalence \eqref{eq:exceptional_descent}, combined with the natural equivalence 
	\begin{equation*}
		p_{n,\sharp}^{\hyp} \simeq \prSlowersharphyp \circ (\id{S} \cross j_{n})_{\sharp}^{\hyp} \period \qedhere
	\end{equation*}
\end{proof}

\begin{notation}
	We denote by
	\begin{equation*} 
		\chi \colon \prSlowersharphyp \circ \prSupperstarhyp \longrightarrow \prSlowerstar \circ \prSupperstarhyp 
	\end{equation*}
	the composition of the counit $\prSlowersharphyp \circ \prSupperstarhyp \to \id{}$ with the unit $\id{} \to \prSlowerstar \circ \prSupperstarhyp$.
\end{notation}

\begin{lemma}\label{lem:computation_exceptional_pushforward}
	In addition to the hypotheses made in \Cref{nul:generic_assumptions}, assume that $X$ is weakly contractible.
	Then for each \smash{$F \in \LChyp_S(S \cross X; \E)$}, the natural transformation $\chi$ induces an equivalence
	\begin{equation*} 
		\equivto{\prSlowersharphyp(F)}{\prSlowerstar(F)} \period 
	\end{equation*}
\end{lemma}

\begin{proof}
	Since $X$ is weakly contractible, \cref{thm:relative_full_faithfulness} guarantees the existence of a hypersheaf $G \in \Shhyp(S; \E)$ and an equivalence \smash{$F \simeq \prSupperstarhyp(G)$}.
	Since \smash{$ \prSupperstarhyp $} is fully faithful (again by \cref{thm:relative_full_faithfulness}), the morphism $ \chi $ applied to $ G $ is the composite equivalence
	\begin{equation*}
		\begin{tikzcd}
			\prSlowersharphyp(F) \equivalent \prSlowersharphyp \prSupperstarhyp(G) \arrow[r, "\sim"{yshift=-0.25em}] & G \arrow[r, "\sim"{yshift=-0.25em}] & \prSlowerstar \prSupperstarhyp(G) \equivalent \prSlowerstar(F) \period 
		\end{tikzcd}
		\qedhere
	\end{equation*}
\end{proof}

\begin{corollary}\label{cor:exceptional_preservation_locally_hyperconstant}
	In addition to the hypotheses made in \Cref{nul:generic_assumptions}, assume that one of the following hypotheses is satisfied:
	\begin{enumerate}[label=\stlabel{cor:exceptional_preservation_locally_hyperconstant}, ref=\arabic*]
		\item\label{cor:exceptional_preservation_locally_hyperconstant.1} The topological space $ X $ is weakly contractible.
		
		\item\label{cor:exceptional_preservation_locally_hyperconstant.2} The topological space $ S $ is locally weakly contractible.
	\end{enumerate}
	Then the functor \smash{$ \prSlowersharphyp \colon \Shhyp(S \cross X; \E) \to \Shhyp(S;\E) $} preserves locally hyperconstant hypersheaves.
\end{corollary}

\begin{proof}
	Let $F \in \LChyp(S \cross X;\E)$.
	To prove the claim under assumption \enumref{cor:exceptional_preservation_locally_hyperconstant}{1}, using \cref{lem:hyperconstant_implies_S_hyperconstant}, we see that $F$ belongs to $\LChyp_S(S \cross X; \E)$.
	\Cref{lem:computation_exceptional_pushforward} implies that
	\begin{equation*}
		\prSlowersharphyp(F) \simeq \prSlowerstar(F) \period
	\end{equation*}
	The conclusion follows from \cref{lem:exceptional_pushforward_locally_hyperconstant}.
	
	To prove the claim under assumption \enumref{cor:exceptional_preservation_locally_hyperconstant}{2}, using \cref{cor:locally_constant_closed_under_limits_colimits} we see that that $\LChyp(S;\E)$ is closed under small colimits in \smash{$\Shhyp(S; \E)$}.
	Using \cref{lem:exceptional_descent}, we can reduce to the case where $X$ is weakly contractible, in which case the conclusion follows from \enumref{cor:exceptional_preservation_locally_hyperconstant}{1}.
\end{proof}


\subsection{Comparison of sheaf and singular cohomology}\label{subsec:cohomology_comparison}

Now we explain why our work implies that for locally weakly contractible spaces, singular and sheaf cohomology agree.

\begin{notation}
	Let $ R $ be a ring and $ X $ a topological space.
	Write $ \Dup(R) $ for the derived \category of $ R $, and write $ \Cup_{*}(X;R) \in \Dup(R) $ for the complex of \textit{singular chains} on $ X $.
	Given an object $ M \in \Dup(R) $, the cotensor $ M^{\Piinfty(X)} $ is given by the internal $ \Hom $ complex
	\begin{equation*}
		\Cup^{-*}(X;M) \colonequals \RHom_R(\Cup_{*}(X;R),M) \period
	\end{equation*}
	If $ M $ is an ordinary $ R $-module, then $ \Cup^{-*}(X;M) $ is what is usually referred to as the complex of \textit{singular cochains} on $ X $ with values in $ M $.
\end{notation}

\begin{nul}
	The functor $ \Pi_{\Dup(R)}^{\infty} \colon \fromto{\Dup(R)}{\Shhyp(X;\Dup(R))} $ is given by the assignment $ \goesto{M}{\Cup^{-*}(-;M)} $.
\end{nul}

The following is an immediate consequence of \Cref{prop:computing_hyperconstant_hypersheaf}:

\begin{corollary}\label{cor:cohomology_comparison}
	Let $ R $ be a ring and $ X $ a locally weakly contractible topological space.
	Then:
	\begin{enumerate}[label=\stlabel{cor:cohomology_comparison}]
		\item The functor $ \fromto{\Dup(R)}{\Shhyp(X;\Dup(R))} $ given by the assignment $ \goesto{M}{\Cup^{-*}(-;M)} $ is the constant hypersheaf functor.

		\item For each $ M \in \Dup(R) $, there is a natural equivalence $ \RGamma(X;M) \equivalence \Cup^{-*}(X;M) $ from the derived global sections of the constant hypersheaf at $ M $ to the complex of singular cochains on $ X $ with values in $ M $.

		\item For each ordinary $ R $-module $ M $, there is a natural isomorphism \smash{$ \Hsheaf^*(X;M) \equivalence \Hsing^*(X;M) $} from sheaf cohomology to singular cohomology.
	\end{enumerate}
\end{corollary}

\noindent Hence, sheaf cohomology is an invariant of the \textit{weak} homotopy type of locally weakly contractible topological spaces. 

\begin{remark}
	After work of Sella \cite{arXiv:1602.06674}, Petersen \cite[Theorem 1.2]{arXiv:2102.06927} recently proved a comparison for cohomology valued in \textit{ordinary} $ R $-modules.
	Petersen's comparison is under slightly weaker assumptions on the topological space $ X $, in relation to the chosen $ R $-module.
	The argument is as follows: given an ordinary $ R $-module $ M $, the morphism \smash{$ \alpha_M \colon \fromto{\GammaXupperstarhyp(M)}{\Cup^{-*}(-;M)} $} is an equivalence if and only if it induces an equivalence on all stalks.
	Given $ x \in X $, the morphism $ \xupperstar \alpha_M $ is an equivalence if and only if the fiber of $ \xupperstar \alpha_M $ is zero. 
	Since singular and sheaf cohomology agree for the point $ \{x\} $, the latter requirement is equivalent to the requirement that for each $ i \geq 0 $, the colimit of relative singular cohomology modules
	\begin{equation*}
		\colim_{U \ni x} \Hsing^i(U,x;M)
	\end{equation*}
	vanishes.
	This vanishing condition is exactly the assumption under which Petersen proves the comparison of sheaf and singular cohomology.
\end{remark}


\section{Homotopy-invariance for locally constant sheaves}\label{sec:homotopy_invariance_for_LC}

The purpose of this section is to prove a non-hypercomplete variant of \Cref{thm:homotopy_invariance_hyperconstant}.
Specifically, for a presentable \category $ \E $ and topological space $ S $, we prove that the pullback functor
\begin{equation*}
	\prSupperstar \colon \incto{\LC(S;\E)}{\LC(S \cross[0,1];\E)}
\end{equation*}
is an equivalence of \categories (\Cref{cor:hiofLC}).
The proof follows the same format of \Cref{thm:homotopy_invariance_hyperconstant} expanding on Clausen and Ørsnes Jansen's proof of \cite[Proposition 3.2]{arXiv:2108.01924}.

\Cref{subsec:exceptional_pushforward_nonhypercomplete} recalls what we need about the exceptional pushforward in the non-hypercomplete setting.
In \cref{subsec:homotopy_invariance_for_LC}, we show that the exceptional pushforward preserves locally constant sheaves; given previous results this immediatly implies that the functor $ \LC(-;\E) $ is homotopy-invariant.

\begin{remark}
	Note that the above homotopy-invariance statement is with respect to the unit interval rather than a general \wclwc space.
	Indeed, we do not expect that $ \LC(-;\E) $ is strongly homotopy-invariant (in the sense of \Cref{def:homotopy_invariant}).
\end{remark}


\subsection{The exceptional pushforward}\label{subsec:exceptional_pushforward_nonhypercomplete}

We now record the existence of the exceptional pushforward in the non-hypercomplete setting as well as its compatibility with basechange.
In this section, we are most interested in the case where $ X $ is a subinterval of $ [0,1] $.

\begin{recollection}\label{rec:locallycompacttensor}
	Let $ S $ and $ X $ be topological spaces.
	There is a natural geometric morphism of \topoi
	\begin{equation*}
		\fromto{\Sh(S \cross X)}{\Sh(S) \tensor \Sh(X)}
	\end{equation*}
	\HA{Example}{4.8.1.19}.
	If $ X $ is locally compact, then this geometric morphism $ \fromto{\Sh(S \cross X)}{\Sh(S) \tensor \Sh(X)} $ is an equivalence \HTT{Proposition}{7.3.1.11}.
\end{recollection}

\begin{lemma}\label{lem:nonhypercomplete_prlowersharp}
	Let $ S $ be a topological space and $ \E $ a presentable \category.
	Let $ X $ be a locally compact topological space and assume that the constant sheaf functor $ \Gammaupperstar_X \colon \fromto{\Spc}{\Sh(X)} $ admits a left adjoint $ \Gamma_{X,\sharp} $.
	Then: 
	\begin{enumerate}[label=\stlabel{lem:nonhypercomplete_prlowersharp}, ref=\arabic*]
		\item The pullback functor $ \prSupperstar \colon \fromto{\Sh(S;\E)}{\Sh(S \cross X;\E)} $ admits a left adjoint $ \prSlowersharp $.

		\item If $ \Gammaupperstar_X \colon \fromto{\Spc}{\Sh(X)} $ is fully faithful, then $ \prSupperstar \colon \fromto{\Sh(S;\E)}{\Sh(S \cross X;\E)} $ is also fully faithful.
	\end{enumerate}
\end{lemma}

\begin{proof}
	Appealing to \Cref{rec:locallycompacttensor}, this follows by tensoring the chain of adjoints
	\begin{equation*}
		\begin{tikzcd}[sep=4em]
			\Sh(X) \arrow[r, shift left=1.5ex, "\Gamma_{X,\sharp}"] \arrow[r, shift right=1.5ex, "\GammaXlowerstar"'] & \Spc \arrow[l, "\GammaXupperstar" description] \period
		\end{tikzcd}
	\end{equation*}
	with the presentable \category $ \Sh(S;\E) $.
\end{proof}

\begin{nul}
	In the situation of \Cref{prop:exceptional_pushforward_sheaves}, we refer to $ \prSlowersharp $ as the \defn{exceptional pushforward}.
\end{nul}

\begin{remark}\label{rmk:prupperstar_CW_complex}
	In light of \Cref{nul:CWhypercomplete,prop:computing_hyperconstant_hypersheaf}, if $ X $ is a topological space that admits a CW structure, then the hypotheses of \Cref{lem:nonhypercomplete_prlowersharp} are satisfied.
	Moreover, if $ X $ is also contractible, then $ \Gammaupperstar_X $ is fully faithful.
\end{remark}

The following compatibility with basechange is immediate from the definition of $ \prSlowersharp $ as a tensor product \cite[Observation 1.15]{arXiv:2108.03545}.
See also \cite[Lemma 3.3]{arXiv:2110.10212}.

\begin{lemma}\label{prop:exceptional_pushforward_sheaves}
	Let $X$ be topological spaces and $ \E $ a presentable \category.
	Given a map $ g \colon T \to S$ of topological spaces, there is a canonically commutative square of \categories
	\begin{equation*}
		\begin{tikzcd}[column sep=4.5em, row sep=2.5em]
			\Sh(S \cross X;\E) \arrow[r, "{(g \cross \id{X})\upperstar}"] \arrow[d, "\prSlowersharp"'] & \Sh(T \cross X;\E) \arrow[d, "\prTlowersharp"] \\ 
			\Sh(S;\E) \arrow[r, "\gupperstar"'] & \Sh(T;\E) \phantom{ \cross X} \period 
		\end{tikzcd}
	\end{equation*}
\end{lemma}

As for hypersheaves, we can check that the unit morphism $ \fromto{F}{\prSupperstar\prSlowersharp(F)} $ is an equivalence locally.

\begin{lemma}\label{lem:pullbackcounit}
	Let $ S $ be a topological space, $ \{f_{\alpha} \colon \fromto{S_{\alpha}}{S}\}_{\alpha \in A} $ a collection of maps of topological spaces, and $ \E $ a presentable \category.
	Assume that the pullback functors
	\begin{equation*}
		\{(f_{\alpha} \cross \id{[0,1]})\upperstar \colon \fromto{\Sh(S \cross [0,1];\E)}{\Sh(S_{\alpha} \cross [0,1];\E)}\}_{\alpha \in A}
	\end{equation*}
	are jointly conservative.
	Given $ F \in \Sh(S \cross [0,1];\E) $, the unit $ \unit_{F} \colon \fromto{F}{\prSupperstar\prSlowersharp(F)} $ is an equivalence if and only if for each $ \alpha \in A $, the unit
	\begin{equation*}
		(f_{\alpha} \cross \id{[0,1]})\upperstar(F) \to \prupperstar_{S_{\alpha}}\pr_{S_{\alpha},\sharp}(f_{\alpha} \cross \id{[0,1]})\upperstar(F)
	\end{equation*}
	is an equivalence.
\end{lemma}

\begin{proof}
	By \Cref{prop:exceptional_pushforward_sheaves}, we have a natural equivalence 
	\begin{equation*}
		(f_{\alpha} \cross \id{[0,1]})\upperstar\prSupperstar\prSlowersharp(F) \equivalent \prupperstar_{S_{\alpha}}\pr_{S_{\alpha},\sharp}(f_{\alpha} \cross \id{[0,1]})\upperstar(F) \period
	\end{equation*}
	Moreover, notice that the pullback
	\begin{equation*}
		(f_{\alpha} \cross \id{[0,1]})\upperstar(\unit_{F}) \colon (f_{\alpha} \cross \id{[0,1]})\upperstar(F) \to \prupperstar_{S_{\alpha}}\pr_{S_{\alpha},\sharp}(f_{\alpha} \cross \id{[0,1]})\upperstar(F)
	\end{equation*}
	is homotopic to the unit of the adjunction $ \pr_{S_{\alpha},\sharp} \leftadjoint \prupperstar_{S_{\alpha}} $ applied to the sheaf $ (f_{\alpha} \cross \id{[0,1]})\upperstar(F) $.
	The claim now follows from the assumption that the functors $ \{(f_{\alpha} \cross \id{[0,1]})\upperstar\}_{\alpha \in A} $ are jointly conservative.
\end{proof}


\subsection{Homotopy-invariance of locally constant sheaves}\label{subsec:homotopy_invariance_for_LC}

We now show that $ \prSlowersharp $ preserves locally constant sheaves.
The compactness of $ [0,1] $ and the fact that $ [0,1] $ has the order topology imply the following:

\begin{lemma}\label{lem:opencover}
	Let $ S $ be a topological space and $ \Ucal $ an open cover of $ S \cross [0,1] $.
	Then there exist:
	\begin{enumerate}[label=\stlabel{lem:opencover}, ref=\arabic*]
		\item An open cover $ \{U_{\alpha}\}_{\alpha \in A} $ of $ S $.

		\item For each $ \alpha \in A $, a positive integer $ n_{\alpha} $ and open subintervals $ I_{\alpha,1}, \ldots, I_{\alpha,n_{\alpha}} $ of $ [0,1] $ covering $ [0,1] $ such that $ I_{\alpha,k} \intersect I_{\alpha,\el} \neq \emptyset $ if and only if $ k = \el \pm 1 $. 
	\end{enumerate}
	Such that $ \Union_{\alpha \in A} \{U_{\alpha} \cross I_{\alpha,1}, \ldots, U_{\alpha} \cross I_{\alpha,n_{\alpha}} \} $ refines the cover $ \Ucal $.
\end{lemma}

\begin{observation}\label{obs:restrictsubint}
	Let $ U $ a topological space, and $ I, J \subset [0,1] $ subintervals which are open in $ [0,1] $.
	Assume that the intersection $ I \intersect J $ is nonempty.
	Since $\{U \cross I, U \cross J\}$ is an open cover of $U \cross (I \union J)$, descent and the fact that the pullback functors
	\begin{equation*}
		\fromto{\Sh(U;\E)}{\Sh(U \cross I;\E)} \, , \quad \fromto{\Sh(U;\E)}{\Sh(U \cross J;\E)} \comma \andeq \fromto{\Sh(U;\E)}{\Sh(U \cross (I \intersect J);\E)}
	\end{equation*}
	are fully faithful (\Cref{lem:nonhypercomplete_prlowersharp,rmk:prupperstar_CW_complex}) implies that if $ F_I \in \Sh(U \cross I;\E) $ and $ F_J \in \Sh(U \cross J;\E) $ are pulled back from $ U $ and 
	\begin{equation*}
		\restrict{F_I}{U \cross (I \intersect J)} \equivalent \restrict{F_J}{U \cross (I \intersect J)} \comma
	\end{equation*}
	then there exists a \textit{unique} sheaf $ G \in \Sh(U;\E) $ such that 
	\begin{equation*}
		F_I \equivalent \prUupperstar(G) \andeq F_J \equivalent \prUupperstar(G) \period
	\end{equation*}
	In particular, if $ L \in \Sh(U \cross (I \union J);\E) $ is such that both $ \restrict{L}{U \cross I} $ and $ \restrict{L}{U \cross J} $ are constant, then $ L $ is constant.
\end{observation}

\begin{lemma}\label{lem:LCcover}
	Let $ S $ be a topological space, and $ L \in \Sh(S \cross [0,1];\E) $.
	If $ L \in \LC(S \cross [0,1];\E) $, then there exists an open cover $ \{U_\alpha\}_{\alpha \in A} $ of $ S $ such that for each $ \alpha \in A $, the sheaf $ \restrict{L}{U_{\alpha} \cross [0,1]} $ is constant.
\end{lemma}

\begin{proof}
	Choose an open cover $\{U_{\alpha} \cross I_{\alpha,1}, \ldots, U_{\alpha} \cross I_{\alpha,n_{\alpha}} \}_{\alpha \in A}$ of $ S \cross [0,1] $ as in \Cref{lem:opencover} such that each restriction $ \restrict{L}{U_{\alpha} \cross I_{\alpha,k}} $ is constant.
	We claim that for each $ \alpha \in A $, the restriction $ \restrict{L}{U_{\alpha} \cross [0,1]} $ is constant.
	To see this, apply \Cref{obs:restrictsubint} inductively with \smash{$ I = I_{\alpha,1} \union \cdots \union I_{\alpha,m} $} and $ J = I_{\alpha,m+1} $.
\end{proof}

\begin{observation}\label{obs:constsheavespushforward}
	Notice that
	\begin{equation*}
		\Gammaupperstar_{S \cross [0,1]} \equivalent \prSupperstar \Gammaupperstar_{S} \period
	\end{equation*}
	Hence, if $ F \in \Sh(S \cross [0,1];\E) $ is constant, then the exceptional pushforward $ \prSlowersharp(F) $ is constant and the unit $ \fromto{F}{\prSupperstar\prSlowersharp(F)} $ is an equivalence.
\end{observation}

\begin{lemma}\label{lem:prlowersharppreservesLC}
	Let $ S $ be a topological space and $ \E $ a presentable \category.
	Then the exceptional pushforward
	\begin{equation*}
		\prSlowersharp \colon \fromto{\Sh(S \cross [0,1];\E)}{\Sh(S;\E)}
	\end{equation*}
	preserves locally constant sheaves.
\end{lemma}

\begin{proof}
	Let $ F \in \LC(S \cross [0,1];\E) $.
	Using \Cref{lem:LCcover}, choose an open cover $ \{U_{\alpha}\}_{\alpha \in A} $ of $ S $ such that each of the restrictions $ \restrict{F}{U_{\alpha} \cross [0,1]} $ is constant. 
	By \Cref{prop:exceptional_pushforward_sheaves} we have $\restrict{\prSlowersharp(F)}{U_{\alpha}} \equivalent \pr_{U_{\alpha},\sharp}(\restrict{F}{U_{\alpha} \cross [0,1]})$.
	Hence \Cref{obs:constsheavespushforward} shows that the sheaf $ \restrict{\prSlowersharp(F)}{U_{\alpha}} $ is constant.
\end{proof}

Homotopy-invariance is now an immediate consequence:

\begin{corollary}\label{cor:hiofLC}
	Let $ S $ be a topological space and $ \E $ a presentable \category.
	Then the functors
	\begin{equation*}
		\adjto{\prSlowersharp}{\LC(S \cross [0,1];\E)}{\LC(S;\E)}{\prSupperstar}
	\end{equation*}
	are inverse equivalences of \categories.
	In particular, the functor $ \LC(-;\E) \colon \fromto{\Top^{\op}}{\Catinfty} $ is homotopy-invariant
\end{corollary}

\begin{proof}
	Since $ \prSupperstar $ is fully faithful, it suffices to show that if $ F \in \LC(S \cross [0,1];\E) $, then the unit \smash{$ \fromto{F}{\prSupperstar\prSlowersharp(F)} $} is an equivalence.
	Using \Cref{lem:LCcover}, choose an open cover $ \{U_{\alpha}\}_{\alpha \in A} $ of $ S $ such that each of the restrictions $ \restrict{F}{U_{\alpha} \cross [0,1]} $ is constant.
	The claim now follows from \Cref{lem:pullbackcounit,obs:constsheavespushforward}.
\end{proof}


\section{Homotopy-invariance for (hyper)constructible (hyper)sheaves}\label{sec:homotopy_invariance_for_constructible}

We now bootstrap our homotopy-invariance results (\cref{thm:homotopy_invariance_hyperconstant,cor:hiofLC}) from the locally constant setting to the constructible setting.
In \cref{subsec:stratifiedspaces}, we review the basics of stratified spaces and (hyper)constructible (hyper)sheaves.
In \cref{subsec:formal_homotopy_invariance}, we prove that the exceptional pushforwards \smash{$ \prSlowersharp $} and \smash{$ \prSlowersharphyp $} preserve constructibility (\Cref{lem:exceptional_preservation_hyperconstructible,cor:prlowersharppreservesCons}) and give equivalent conditions for homotopy-invariance to hold (\Cref{cor:formal_homotopy_invariance}).
Finally, in \cref{subsec:detecting_equivalences_on_strata} we use these criteria to show that, in many situations of interest, (hyper)constructible (hyper)sheaves are homotopy-invariant (\Cref{cor:hiofConsCG,cor:truncated,cor:hiofConsnoeth}).


\subsection{Stratified spaces \& constructible sheaves}\label{subsec:stratifiedspaces}

We first recall the notion of a stratified space:

\begin{notation}\label{notation:Alexandroff_topology}
	Let $ P $ be a poset.
	We also write $ P $ for the set $ P $ equipped with the \textit{Alexandroff topology} in which a subset $ U \subset P $ is open if and only if $ U $ is upwards-closed.
	Given an element $ p \in P $, we write
	\begin{equation*}
		P_{\geq p} \colonequals \{q \in P \, | \, q \geq p \}  \andeq P_{>p} \colonequals P_{\geq p} \sminus \{p\} \period
	\end{equation*}
	The category of \defn{$ P $-stratified topological spaces} is the overcategory \smash{$ \TopP $}.
	Given a $ P $-stratified topological space $ \sigma \colon \fromto{S}{P} $ and $ p \in P $, we write \smash{$ S_p \colonequals \sigmainverse(p) $} and call $ S_p $ the \defn{$ p $-th stratum} of $ S $.
	We also write
	\begin{equation*}
		S_{\geq p} \colonequals \sigmainverse(P_{\geq p}) \andeq S_{> p} \colonequals \sigmainverse(P_{> p}) \period
	\end{equation*}
	We write $ i_p \colon \fromto{S_p}{S} $ for the inclusion of the $ p $-th stratum.
\end{notation}

\begin{definition}\label{def:constructible}
	Let $ P $ be a poset, $ \fromto{S}{P} $ be a $ P $-stratified space, and $ \E $ be a presentable \category.
	\begin{enumerate}[label=\stlabel{def:constructible}, ref=\arabic*]
		\item We say that a sheaf $ F \in \Sh(S;\E) $ is a \defn{$ P $-constructible} if $ F $ for each $ p \in P $, the restriction $ \iupperstar_p(F) $ is a locally constant sheaf on the stratum $ S_p $.
		
		\item We say that a hypersheaf $ F \in \Shhyp(S;\E) $ is a \defn{$ P $-hyperconstructible} if $ F $ for each $ p \in P $, the restriction \smash{$ \iupperstarhyp_p(F) $} is a locally hyperconstant hypersheaf on the stratum $ S_p $.
	\end{enumerate}
	We, respectively, write
	\begin{equation*}
		\ConsP(T;\E) \subset \Sh(T;\E) \andeq \ConsPhyp(T;\E) \subset \Shhyp(T;\E)
	\end{equation*}
	for the full subcategories spanned by the $ P $-constructible sheaves and $ P $-hyperconstructible hypersheaves.
\end{definition}

\begin{warning}
	There is a containment
	\begin{equation*}
		\ConsP(S;\E) \intersect \Shhyp(S;\E) \subset \ConsPhyp(S;\E) \comma
	\end{equation*}
	however, this inclusion need not be an equality.
	Also note that if $ F $ is a $ P $-constructible sheaf, then then $ F^{\hyp} \in \ConsPhyp(S;\E) $.
\end{warning}

\begin{remark}\label{rem:Conshyp_and_Cons_agree}
	Let $ P $ be a Noetherian poset and let $ X \to P $ be a paracompact $ P $-stratified space.
	Assume that the stratification of $ X $ is \textit{conical} in the sense of \HA{Definition}{A.5.5} and that all of the strata of $ X $ are locally of singular shape.
	Then
	\begin{equation*}
		\ConsPhyp(X) = \ConsP(X) \intersect \Shhyp(X) = \ConsP(X) \period
	\end{equation*}
	See \cites[\HAappthm{Proposition}{A.5.9}]{HA}[Proposition 2.11]{arXiv:2102.12325}.
\end{remark}

\begin{observation}\label{obs:pullbackofhypcons}
	For any map $ f \colon \fromto{T}{S} $ of $ P $-stratified spaces, the sheaf pullback functor $\fupperstar $ preserves $ P $-constructible sheaves and the hypersheaf pullback functor $\fupperstarhyp $ preserves $ P $-hyperconstructible hypersheaves.
	Hence the assignments
	\begin{equation*}
		\goesto{S}{\ConsP(S;\E)} \andeq \goesto{S}{\ConsPhyp(S;\E)}
	\end{equation*}
	define subfunctors of the functors $ \Sh(-;\E),\Shhyp(-;\E) \colon \fromto{\TopP^{\op}}{\Catinfty} $.
\end{observation}


\begin{convention}
	Let $ P $ be a poset and $ \sigma \colon \fromto{S}{P} $ be a $ P $-stratified topological space.
	Let $X$ be a topological space.
	We write $ S \cross X $ for the $ P $-stratified topological space with stratification given by the composite
	\begin{equation*}
		\begin{tikzcd}
			S \cross X \arrow[r, "\pr_{S}"] & S \arrow[r, "\sigma"] & P \period
		\end{tikzcd}
	\end{equation*}
\end{convention}


The main goal of this section is to explain when the functors $\ConsP(-;\E) $, and \smash{$ \ConsPhyp(-;\E) $} are homotopy-invariant in the sense of \cref{def:homotopy_invariant}.
Again, we remind the reader that a functor \smash{$ C \colon \fromto{\TopP^{\op}}{\Catinfty} $} is homotopy-invariant in the sense of \Cref{def:homotopy_invariant} if and only if $ C $ inverts all \textit{stratified homotopy equivalences} in the following sense.

\begin{definition}[(stratified homotopy)]\label{def:stratified_homotopies}
	Let $ P $ be a poset.
	\begin{enumerate}[label=\stlabel{def:stratified_homotopies}, ref=\arabic*]
		\item Given maps of $ P $-stratified topological spaces $ f_0,f_1 \colon \fromto{T}{S} $, a \defn{$ P $-stratified homotopy} from $ f_0 $ to $ f_1 $ is a $ P $-stratified map $ h \colon \fromto{T \cross [0,1]}{S} $ such that $ \restrict{h}{T \cross \{0\}} = f_0 $ and $ \restrict{h}{T \cross \{1\}} = f_1 $.

		\item A map of $ P $-stratified topological spaces $ f \colon \fromto{T}{S} $ is a \defn{$ P $-stratified homotopy equivalence} if there exists a $ P $-stratified map $ g \colon \fromto{S}{T} $ and $ P $-stratified homotopies from $ gf $ to $ \id{T} $ and from $ fg $ to $ \id{S} $.
	\end{enumerate}
\end{definition}


\begin{lemma}\label{lem:cons_homotopy_invariance_reformulation}
	Let $ P $ be a poset.
	The following are equivalent for a functor \smash{$ C \colon \fromto{\TopP}{\Catinfty} $}:
	\begin{enumerate}[label=\stlabel{lem:cons_homotopy_invariance_reformulation}, ref=\arabic*]
		\item The functor $ C $ is homotopy-invariant.

		\item For each $ P $-stratified homotopy equivalence $ f \colon \fromto{T}{S} $, the functor $ C(f) \colon \fromto{C(S)}{C(T)} $ is an equivalence of \categories.
	\end{enumerate}
\end{lemma}


\subsection{Formal homotopy-invariance}\label{subsec:formal_homotopy_invariance}

Bootstrapping off of the results of \cref{sec:hyperconstant,sec:homotopy_invariance_for_LC}, we can provide a first, formal version of our homotopy-invariance result.
The compatibility of the exceptional pushforwards with pullbacks immediately gives:

\begin{corollary}\label{lem:exceptional_preservation_hyperconstructible}
	Let $\E$ be a presentable \category and let $P$ be a poset.
	Let $S $ be a $ P $-stratified space and let $X$ be a \wclwc topological space.
	Then the exceptional hypersheaf pushforward
	\begin{equation*} 
		\prSlowersharphyp \colon \Shhyp(S \cross X; \E) \longrightarrow \Shhyp(S; \E)
	\end{equation*}
	preserves hyperconstructible hypersheaves.
\end{corollary}

\begin{proof}
	Let $F \in \ConsPhyp(S \cross X; \E)$ be a hyperconstructible hypersheaf on $S \cross X$.
	We have to prove that for every $p \in P$, the restriction
	\begin{equation*}
		\iupperstarhyp_p\prSlowersharphyp(F)
	\end{equation*}
	is a locally hyperconstant hypersheaf on $ S_p $.
	By the compatibility of the exceptional pushforward with pullbacks (\cref{cor:projection_left_adjointable}), there is a natural equivalence
	\begin{equation*} 
		\iupperstarhyp_p(\prSlowersharphyp(F)) \simeq \pr_{S_p,\sharp}^{\hyp}(( i_p \cross \id{X} )^{*,\hyp}(F)) \period 
	\end{equation*}
	Since $(i_p \cross \id{X})^{*,\hyp}(F)$ is locally hyperconstant on $S_p \cross X$, \cref{cor:exceptional_preservation_locally_hyperconstant} completes the proof.
\end{proof}

\begin{corollary}\label{cor:prlowersharppreservesCons}
	For $ S \in \TopP $, the functor $ \prSlowersharp \colon \fromto{\Sh(S \cross [0,1];\E)}{\Sh(S;\E)} $ preserves constructible sheaves.
\end{corollary}

\begin{proof}
	As in the proof of \Cref{lem:exceptional_preservation_hyperconstructible}, combine \Cref{lem:prlowersharppreservesLC} with \Cref{prop:exceptional_pushforward_sheaves}.
\end{proof}


\begin{theorem} \label{thm:formal_homotopy_invariance}
	Under the hypotheses of \cref{lem:exceptional_preservation_hyperconstructible}, the essential image of the fully faithful functor
	\begin{equation}\label{eq:prSupperstarhyp_Cons}
		\prSupperstarhyp \colon \ConsPhyp(S;\E) \inclusion \ConsPhyp(S \cross X; \E) 
	\end{equation}
	is the intersection $\ConsPhyp(S \cross X; \E) \intersect \LChyp_S(S \cross X; \E)$.
\end{theorem}

\begin{proof}
	Since
	\begin{equation*} 
		\prSupperstarhyp \colon \Shhyp(S; \E) \longrightarrow \LChyp_S(S \cross X; \E)
	\end{equation*}
	is an equivalence of \categories and preserves constructiblity, we immediately see that the essential image of \eqref{eq:prSupperstarhyp_Cons} is contained in
	\begin{equation*}
		\ConsPhyp(S \cross X; \E) \intersect \LChyp_S(S \cross X; \E) \period
	\end{equation*}
	Conversely, assume that $F$ belongs to this intersection.
	Since \smash{$F \in \LChyp_S(S \cross X; \E)$}, \cref{thm:relative_full_faithfulness,lem:computation_exceptional_pushforward} imply that
	\begin{equation*}
		F \simeq \prSupperstarhyp( \prSlowersharphyp(F))
	\end{equation*}
	Since $F$ belongs to $\ConsPhyp(S \cross X; \E)$, \cref{lem:exceptional_preservation_hyperconstructible} implies that \smash{$\prSlowersharphyp(F) \in \ConsPhyp(S;\E)$}.
	Therefore, $F$ belongs to the essential image of \eqref{eq:prSupperstarhyp_Cons}, as desired.
\end{proof}

\begin{corollary}\label{cor:formal_homotopy_invariance}
	Under the hypotheses of \cref{lem:exceptional_preservation_hyperconstructible}, the following conditions are equivalent:
	\begin{enumerate}[label=\stlabel{cor:formal_homotopy_invariance}, ref=\arabic*]
		\item\label{cor:formal_homotopy_invariance.1} The functor $ \prSupperstarhyp \colon \ConsPhyp(S;\E) \inclusion \ConsPhyp(S \cross X; \E) $ is an equivalence.
		
		\item\label{cor:formal_homotopy_invariance.2} For each $F \in \ConsPhyp(S \cross X; \E)$, the unit $ F \to \prSupperstarhyp( \prSlowersharphyp(F))  $ is an equivalence.
		
		\item\label{cor:formal_homotopy_invariance.3} We have $\ConsPhyp(S\cross X; \E) \subset \LChyp_S(S\cross X; \E)$ as subcategories of $ \Shhyp(S\cross X; \E) $.
		
		\item\label{cor:formal_homotopy_invariance.4} For each $F \in \ConsPhyp(S \cross X; \E)$, each open subset $W \subset S$, and each pair of weakly contractible open subsets $U \subset V$ of $X$, the restriction map $ F(W \cross V) \to F(W \cross U) $ is an equivalence.
	\end{enumerate}
\end{corollary}

\begin{proof}
	The equivalence between \enumref{cor:formal_homotopy_invariance}{1}, \enumref{cor:formal_homotopy_invariance}{2}, and \enumref{cor:formal_homotopy_invariance}{3} follows from \cref{thm:formal_homotopy_invariance}.
	On the other hand, \cref{prop:strong_local_hyperconstancy} shows that \enumref{cor:formal_homotopy_invariance}{3} and \enumref{cor:formal_homotopy_invariance}{4} are equivalent.
\end{proof}

In particular, we obtain the following sufficient criterion ensuring that $\prSupperstarhyp$ is an equivalence:

\begin{corollary}\label{cor:sufficient_criterion_for_homotopy_invariance}
	In the situation of \cref{thm:formal_homotopy_invariance}, assume that the hypersheaf restriction functors
	\begin{equation*} 
		\left\{\restrict{(-)}{S_p \cross X}^{\hyp} \colon \ConsPhyp(S \cross X; \E) \longrightarrow \LChyp(S_p \cross X; \E)\right\}_{p\in P} 
	\end{equation*}
	are jointly conservative.
	Then the functors
	\begin{equation*} 
		\adjto{\prSlowersharphyp}{\ConsPhyp(S \cross X;\E)}{\ConsPhyp(S;\E)}{\prSupperstarhyp}
	\end{equation*}
	are inverse equivalences.
\end{corollary}

\begin{proof}
	Combine \cref{thm:homotopy_invariance_hyperconstant}, \Cref{cor:pullbackcounit}, and \enumref{cor:formal_homotopy_invariance}{2}.
\end{proof}

\begin{corollary}\label{cor:sufficient_criterion_for_homotopy_invariance_nonhypercomplete}
	Let $ P $ be a poset, $ S \in \TopP $, and $ \E $ a presentable \category.
	Assume that the restriction functors
	\begin{equation*} 
		\left\{\restrict{(-)}{S_p \cross [0,1]} \colon \ConsP(S \cross [0,1]; \E) \longrightarrow \LC(S_p \cross [0,1]; \E)\right\}_{p\in P} 
	\end{equation*}
	are jointly conservative.
	Then the functors
	\begin{equation*} 
		\adjto{\prSlowersharp}{\ConsP(S \cross [0,1];\E)}{\ConsP(S;\E)}{\prSupperstar}
	\end{equation*}
	are inverse equivalences.
\end{corollary}

\begin{proof}
	Combine \Cref{lem:pullbackcounit,cor:hiofLC,cor:prlowersharppreservesCons}.
\end{proof}


\subsection{Detecting equivalences on strata}\label{subsec:detecting_equivalences_on_strata}

\Cref{cor:sufficient_criterion_for_homotopy_invariance} shows that a sufficient criterion for the functor
\begin{equation*}
	\prSupperstarhyp \colon \ConsPhyp(S; \E) \longrightarrow \ConsPhyp(S \cross X; \E) 
\end{equation*}
to be an equivalence is given by the joint conservativity of the hyperrestrictions to the strata of $S \cross X$.
We offer two ways of checking this independently of both $S$ and $X$.


\subsubsection{The compactly generated case}
\hfill

The fact that equivalences of hypersheaves on a topological space with values in a compactly generated \category can be checked on stalks implies our first homotopy-invariance result:

\begin{corollary}\label{lem:CGconservativity}\label{cor:hiofConsCG}
	Let $ P $ be a poset, $ S \in \TopP $, and let $ \E $ be a compactly generated \category.
	Then:
	\begin{enumerate}[label=\stlabel{lem:CGconservativity}, ref=\arabic*]
		\item\label{lem:CGconservativity.1} The restriction functors \smash{$ \{\restrict{(-)}{S_p}^{\hyp} \colon \fromto{\Shhyp(S;\E)}{\Shhyp(S_p;\E)}\}_{p \in P} $} are jointly conservative.

		\item\label{lem:CGconservativity.2} The functor $ \ConsPhyp(-;\E) \colon \TopP^{\op} \longrightarrow \Cat_\infty $ is strongly homotopy-invariant.
	\end{enumerate}
\end{corollary}

\begin{proof}
	\Cref{rec:stalksconservativeCGhyp} immediately implies \enumref{lem:CGconservativity}{1}.
	For \enumref{lem:CGconservativity}{2}, combine \enumref{lem:CGconservativity}{1} for the $ P $-stratified space $ S \cross X $ with \cref{cor:sufficient_criterion_for_homotopy_invariance}.
\end{proof}


\begin{notation}
	Let $ P $ be a poset and $ \fromto{S}{P} $ be a $ P $-stratified topological space.
	Write
	\begin{equation*}
		\ConsP(S)_{<\infty} \subset \ConsPhyp(S)
	\end{equation*}
	for the full subcategory spanned by those $ P $-constructible sheaves that are also $ n $-truncated for some integer $ n \geq 0 $.
	Since left exact functors preserve truncated objects, the assignment $ \goesto{S}{\ConsP(S)_{<\infty}} $ defines a subfunctor of $ \ConsPhyp $.
\end{notation}	

\begin{corollary}\label{cor:truncated}
	Let $ P $ be a poset.
	The functor \smash{$ \ConsP(-)_{<\infty} \colon \fromto{\TopP^{\op}}{\Catinfty} $} is strongly homotopy-invariant.
\end{corollary}


\subsubsection{The Noetherian case}
\hfill

In order to drop the compact generation assumption on $\E$, there are two difficulties to overcome.
Recall that a poset $ P $ is \defn{Noetherian} if $ P $ satisfies the \defn{ascending chain condition}: there does not exist an infinite strictly ascending sequence $ p_0 < p_1 < p_2 < \cdots $ of elements of $ P $.
The first issue is that there exist non-Noetherian posets $ P $ for which the \topos $ \Sh(P) = \ConsP(P) $ is not hypercomplete; see \cite[Example A.13]{arXiv:2001.00319}.
Said differently, the functors $ \fromto{\Sh(P)}{\Sh(\{p\})} $ given by pulling back to strata need not be jointly conservative.
Thus we restrict ourselves to Noetherian posets.

The second issue is with the coefficient \category $ \E $.
Consider the most simple stratification when $ P = \{0 < 1\} $, so that a stratification $ \fromto{S}{\{0 < 1\}} $ is the data of a closed subspace $ Z = S_0 $ and its open complement $ S \sminus Z = S_1 $.
Unfortunately, in general the restriction functors
\begin{equation*}
	\restrict{(-)}{Z} \colon \fromto{\Sh(S;\E)}{\Sh(Z;\E)} \andeq \restrict{(-)}{S \sminus Z} \colon \fromto{\Sh(S;\E)}{\Sh(S \sminus Z;\E)}
\end{equation*}
need not be jointly conservative.
Thus, we have to assume this property:

\begin{definition}\label{def:respectsgluing}
	We say that a presentable \category $ \E $ \defn{respects gluing} if for each topological space $ S $ and closed subspace $ Z \subset S $, the restriction functors
	\begin{equation*}
		\restrict{(-)}{Z} \colon \fromto{\Sh(S;\E)}{\Sh(Z;\E)} \andeq \restrict{(-)}{S \sminus Z} \colon \fromto{\Sh(S;\E)}{\Sh(S \sminus Z;\E)}
	\end{equation*}
	and the hypersheaf restriction functors
	\begin{equation*}
		\restrict{(-)}{Z}^{\hyp} \colon \fromto{\Shhyp(S;\E)}{\Shhyp(Z;\E)} \andeq \restrict{(-)}{S \sminus Z} \colon \fromto{\Shhyp(S;\E)}{\Shhyp(S \sminus Z;\E)}
	\end{equation*}
	are jointly conservative.
\end{definition}

\noindent Luckily, many presentable \categories that arise in nature respect gluing:

\begin{example}
	If each \category $ \Sh(S;\E) $ is the \textit{recollement} of $ \Sh(Z;\E) $ and $ \Sh(S \sminus Z;\E) $ in the sense of \HAa{Definition}{A.8.1}, and each \category \smash{$ \Shhyp(S;\E) $} is the recollement of \smash{$ \Shhyp(Z;\E) $} and \smash{$ \Shhyp(S \sminus Z;\E) $}, then $ \E $ respects gluing.
	Importantly, this is satisfied if $ \E $ is stable or $ \E \equivalent \C \tensor \D $ where $ \C $ is a compactly generated \category and $ \D $ is \atopos \cite[Corollary 2.13, Proposition 2.21, \& Remark 2.26]{arXiv:2108.03545}.
\end{example}

\begin{lemma}\label{lem:Noetherianconservativity}
	Let $ P $ be a Noetherian poset, $ S \in \TopP $, and let $ \E $ be a presentable \category that respects gluing.
	Then:
	\begin{enumerate}[label=\stlabel{lem:Noetherianconservativity}, ref=\arabic*]
		\item\label{lem:Noetherianconservativity.1} The pullback functors \smash{$ \{\restrict{(-)}{S_p} \colon \fromto{\Sh(S;\E)}{\Sh(S_p;\E)}\}_{p \in P} $} are jointly conservative.

		\item\label{lem:Noetherianconservativity.2} The hypersheaf pullback functors \smash{$ \{\restrict{(-)}{S_p}^{\hyp} \colon \fromto{\Shhyp(S;\E)}{\Shhyp(S_p;\E)}\}_{p \in P} $} are jointly conservative.
	\end{enumerate}
\end{lemma}

\begin{proof}
	We prove \enumref{lem:Noetherianconservativity}{1}; the proof of \enumref{lem:Noetherianconservativity}{2} is exactly the same, replacing sheaves by hypersheaves.
	Let $ \phi $ be a morphism in $ \Sh(S;\E) $ that restricts to an equivalence on each stratum; we need to show that $ \phi $ is an equivalence.
	Since the open subsets $ \{S_{\geq p}\}_{p \in P} $ cover $ S $, it suffices to show:
	\begin{enumerate}
		\item[($\ast$)]\label{item:noethinductionclaim} For each $ p \in P $, the restriction $ \restrict{\phi}{S_{\geq p}} $ is an equivalence in $ \Sh(S_{\geq p};\E) $.
	\end{enumerate}
	We prove (\hyperref[item:noethinductionclaim]{$ \ast $}) by Noetherian induction on $ p \in P $.
	We need to show that if the restriction \smash{$ \restrict{\phi}{S_{\geq q}} $} is an equivalence for each $ q > p $, then \smash{$ \restrict{\phi}{S_{\geq p}} $} is an equivalence.
	Note that
	\begin{equation*}
		S_{\geq p} \sminus S_p = S_{>p} = \Union_{q \in P_{>p}} S_{\geq q} \period
	\end{equation*}
	Hence the inductive hypothesis implies that the restriction \smash{$ \restrict{\phi}{S_{>p}} $} is an equivalence.
	By assumption \smash{$ \restrict{\phi}{S_{p}} $} is also an equivalence.
	Since $ \E $ respects gluing, the restriction functors
	\begin{equation*}
		\restrict{(-)}{S_p} \colon \fromto{\Sh(S_{\geq p};\E)}{\Sh(S_p;\E)} \andeq \restrict{(-)}{S_{>p}} \colon \fromto{\Sh(S_{\geq p};\E)}{\Sh(S_{>p};\E)}
	\end{equation*}
	are jointly conservative, completing the proof.
\end{proof}

Finally we deduce the homotopy-invariance of constructible sheaves.

\begin{corollary}\label{cor:hiofConsnoeth}
	Let $ P $ be a Noetherian poset and let $ \E $ be a presentable \category that respects gluing.
	Then:
	\begin{enumerate}[label=\stlabel{cor:hiofConsnoeth}, ref=\arabic*]
		\item\label{cor:hiofConsnoeth.1} The functor \smash{$ \ConsP(-;\E) \colon \fromto{\TopP^{\op}}{\Catinfty} $} is homotopy-invariant

		\item\label{cor:hiofConsnoeth.2} The functor \smash{$ \ConsPhyp(-;\E) \colon \fromto{\TopP^{\op}}{\Catinfty} $} is strongly homotopy-invariant.
	\end{enumerate}
\end{corollary}

\begin{proof}
	Combine \Cref{cor:sufficient_criterion_for_homotopy_invariance,cor:sufficient_criterion_for_homotopy_invariance_nonhypercomplete} with \Cref{lem:Noetherianconservativity}.
\end{proof}


\DeclareFieldFormat{labelnumberwidth}{#1}
\printbibliography[keyword=alph]
\DeclareFieldFormat{labelnumberwidth}{{#1\adddot\midsentence}}
\printbibliography[heading=none, notkeyword=alph]

\end{document}

%% file: newcommands_all.tex
\ifpersonal
\newcommand{\personal}[1]{\textcolor[rgb]{0,0,1}{(Personal: #1)}}
\newcommand{\discussion}[1]{\textcolor{violet}{(Discussion: #1)}}
\else
\newcommand{\personal}[1]{\ignorespaces}
\newcommand{\discussion}[1]{\ignorespaces}
\fi

\newcommand{\cB}{\mathcal B}
\newcommand{\cC}{\mathcal C}

\newcommand{\cE}{\mathcal E}

\DeclareFontFamily{U}{BOONDOX-calo}{\skewchar\font=45 }
\DeclareFontShape{U}{BOONDOX-calo}{m}{n}{<-> s*[1.05] BOONDOX-r-calo}{}
\DeclareFontShape{U}{BOONDOX-calo}{b}{n}{<-> s*[1.05] BOONDOX-b-calo}{}
\DeclareMathAlphabet{\mathcalboondox}{U}{BOONDOX-calo}{m}{n}

\makeatletter
\let\save@mathaccent\mathaccent
\newcommand*\if@single[3]{%
	\setbox0\hbox{${\mathaccent"0362{#1}}^H$}%
	\setbox2\hbox{${\mathaccent"0362{\kern0pt#1}}^H$}%
	\ifdim\ht0=\ht2 #3\else #2\fi
}
\newcommand*\rel@kern[1]{\kern#1\dimexpr\macc@kerna}
\newcommand*\widebar[1]{\@ifnextchar^{{\wide@bar{#1}{0}}}{\wide@bar{#1}{1}}}
\newcommand*\wide@bar[2]{\if@single{#1}{\wide@bar@{#1}{#2}{1}}{\wide@bar@{#1}{#2}{2}}}
\newcommand*\wide@bar@[3]{%
	\begingroup
	\def\mathaccent##1##2{%
		\let\mathaccent\save@mathaccent
		\if#32 \let\macc@nucleus\first@char \fi
		\setbox\z@\hbox{$\macc@style{\macc@nucleus}_{}$}%
		\setbox\tw@\hbox{$\macc@style{\macc@nucleus}{}_{}$}%
		\dimen@\wd\tw@
		\advance\dimen@-\wd\z@
		\divide\dimen@ 3
		\@tempdima\wd\tw@
		\advance\@tempdima-\scriptspace
		\divide\@tempdima 10
		\advance\dimen@-\@tempdima
		\ifdim\dimen@>\z@ \dimen@0pt\fi
		\rel@kern{0.6}\kern-\dimen@
		\if#31
		\overline{\rel@kern{-0.6}\kern\dimen@\macc@nucleus\rel@kern{0.4}\kern\dimen@}%
		\advance\dimen@0.4\dimexpr\macc@kerna
		\let\final@kern#2%
		\ifdim\dimen@<\z@ \let\final@kern1\fi
		\if\final@kern1 \kern-\dimen@\fi
		\else
		\overline{\rel@kern{-0.6}\kern\dimen@#1}%
		\fi
	}%
	\macc@depth\@ne
	\let\math@bgroup\@empty \let\math@egroup\macc@set@skewchar
	\mathsurround\z@ \frozen@everymath{\mathgroup\macc@group\relax}%
	\macc@set@skewchar\relax
	\let\mathaccentV\macc@nested@a
	\if#31
	\macc@nested@a\relax111{#1}%
	\else
	\def\gobble@till@marker##1\endmarker{}%
	\futurelet\first@char\gobble@till@marker#1\endmarker
	\ifcat\noexpand\first@char A\else
	\def\first@char{}%
	\fi
	\macc@nested@a\relax111{\first@char}%
	\fi
	\endgroup
}
\makeatother



%% file: document_macros.tex
\tikzcdset{
  cells={font=\everymath\expandafter{\the\everymath\displaystyle}},
}

\renewcommand{\longrightarrow}{\to}

\renewcommand{\Open}{\mathrm{Open}}
\renewcommand{\Cup}{\mathrm{C}}
\newcommand{\wclwc}{wclwc\xspace}
\renewcommand{\ni}{\smallni}

\DeclareMathOperator{\ctr}{ctr}
\DeclareMathOperator{\all}{all}

\DeclareMathOperator{\sing}{sing}
\DeclareMathOperator{\sheaf}{sheaf}

\newcommand{\RGamma}{\Rup\Gamma}
\DeclareMathOperator{\RHom}{RHom}

\newcommand{\Hsing}{\Hup_{\sing}}
\newcommand{\Hsheaf}{\Hup_{\sheaf}}

\newcommand{\Piinfty}{\Pi_{\infty}}
\newcommand{\PiEinfty}{\Pi_{\E}^{\infty}}
\newcommand{\PiinftyE}{\Pi_{\infty}^{\E}}
\newcommand{\Catinfty}{\Cat_{\infty}}

\DeclareMathOperator{\Cons}{Cons}
\newcommand{\ConsP}{\Cons_{P}}
\newcommand{\ConsPhyp}{\ConsP^{\hyp}}

\DeclareMathOperator{\Env}{Env}

\newcommand{\LChyp}{\LC^{\hyp}}
\newcommand{\TopP}{\Top_{\!/P}}

\renewcommand{\C}{\mathcal{C}}
\renewcommand{\D}{\mathcal{D}}
\renewcommand{\E}{\mathcal{E}}

\newcommand{\allctr}{\all,\ctr}
\newcommand{\Openctr}{\Open_{\ctr}}
\newcommand{\Openctrall}{\Open_{\allctr}}
\newcommand{\Opencross}{\Open_{\cross}}

\DeclareMathOperator{\lwc}{lwc}
\newcommand{\Toplwc}{\Top^{\lwc}}
\newcommand{\Toplwcop}{\Top^{\lwc,\op}}

\newcommand{\prSlowersharp}{\pr_{S,\sharp}}
\newcommand{\prTlowersharp}{\pr_{T,\sharp}}

\newcommand{\prslowersharphyp}{\pr_{\{s\},\sharp}^{\hyp}}

\newcommand{\prSlowersharphyp}{\pr_{S,\sharp}^{\hyp}}
\newcommand{\prTlowersharphyp}{\pr_{T,\sharp}^{\hyp}}

\newcommand{\prXlowerstar}{\pr_{X,\ast}}
\newcommand{\prSlowerstar}{\pr_{S,\ast}}
\newcommand{\prTlowerstar}{\pr_{T,\ast}}

\newcommand{\prSupperstar}{\prupperstar_{S}}

\newcommand{\prUupperstar}{\prupperstar_{U}}

\newcommand{\prupperstarhyp}{\pr^{\ast,\hyp}}
\newcommand{\prsupperstarhyp}{\prupperstarhyp_{\{s\}}}
\newcommand{\prSupperstarhyp}{\prupperstarhyp_{S}}
\newcommand{\prTupperstarhyp}{\prupperstarhyp_{T}}

\newcommand{\GammaXlowerstar}{\Gamma_{X,\ast}}

\newcommand{\GammaXlowersharphyp}{\Gamma_{X,\sharp}^{\hyp}}

\newcommand{\upperstarhyp}{^{\ast,\hyp}}
\newcommand{\fupperstarhyp}{f\upperstarhyp}
\newcommand{\iupperstarhyp}{i\upperstarhyp}
\newcommand{\gupperstarhyp}{g\upperstarhyp}

\newcommand{\GammaXupperstar}{\Gammaupperstar_{X}}

\newcommand{\Gammaupperstarhyp}{\Gamma\upperstarhyp}

\newcommand{\GammaXupperstarhyp}{\Gammaupperstarhyp_{X}}

\newcommand{\inv}{^{-1}}

\newcommand{\sigmainverse}{\sigma\inv}
\newcommand{\finverse}{f\inv}

\newcommand{\coloneqq}{\coloneq}

\theoremstyle{definition}

\makeatletter
\def\@tocline#1#2#3#4#5#6#7{\relax
  \ifnum #1>\c@tocdepth 
  \else
    \par \addpenalty\@secpenalty\addvspace{#2}%
    \begingroup \hyphenpenalty\@M
    \@ifempty{#4}{%
      \@tempdima\csname r@tocindent\number#1\endcsname\relax
    }{%
      \@tempdima#4\relax
    }%
    \parindent\z@ \leftskip#3\relax \advance\leftskip\@tempdima\relax
    \rightskip\@pnumwidth plus4em \parfillskip-\@pnumwidth
    #5\leavevmode\hskip-\@tempdima
      \ifcase #1
       \or\or \hskip 1em \or \hskip 2em \else \hskip 3em \fi%
      #6\nobreak\relax
    \hfill\hbox to\@pnumwidth{\@tocpagenum{#7}}\par
    \nobreak
    \endgroup
  \fi}
\makeatother